\documentclass[11pt]{article}

\usepackage{amsmath,amsthm,amsfonts,amssymb,mathrsfs,bm,wasysym}
\usepackage{epsfig}
\usepackage[usenames]{color}
\usepackage{verbatim}

\usepackage{multicol}
\usepackage{comment}
\usepackage{float}
\usepackage{graphicx}
\usepackage{showlabels}
\usepackage{centernot}
\usepackage{tikz}
\usepackage{color}
\usepackage[]{algorithm2e}
\usepackage{enumerate}

\graphicspath{{./Pics}}
\usepackage[textsize=footnotesize, color=green]{todonotes}
\usepackage[normalem]{ulem}

\parskip \medskipamount
\newtheorem{theorem}{Theorem}[section]
\newtheorem{definition}[theorem]{Definition}

\numberwithin{equation}{section}
\newtheorem{lemma}[theorem]{Lemma}
\newtheorem{proposition}[theorem]{Proposition}
\newtheorem{remark}[theorem]{Remark}

\newtheorem{claim}[theorem]{Claim}

\newcommand{\Var}{{\mathrm{Var}}}

\numberwithin{equation}{section}

\def\N{\mathbb{N}}

\def\F{\mathcal{F}}

\def\B{\mathcal{B}}

\renewcommand{\phi}{\varphi}
\renewcommand{\epsilon}{\varepsilon}

\allowdisplaybreaks

\newcommand{\1}{{\text{\Large $\mathfrak 1$}}}

\newcommand{\cov}{\operatorname{Cov}}

\renewcommand{\emptyset}{\varnothing}

\newcommand{\til}{\widetilde}

\newcommand{\trel}{t_{\mathrm{rel}}}

\newcommand{\pr}[1]{\mathbb{P}\!\left(#1\right)}
\newcommand{\E}[1]{\mathbb{E}\!\left[#1\right]}
\newcommand{\estart}[2]{\mathbb{E}_{#2}\!\left[#1\right]}
\newcommand{\prstart}[2]{\mathbb{P}_{#2}\!\left(#1\right)}
\newcommand{\prcond}[3]{\mathbb{P}_{#3}\!\left(#1\;\middle\vert\;#2\right)}
\newcommand{\econd}[2]{\mathbb{E}\!\left[#1\;\middle\vert\;#2\right]}

\def\dx{\mathrm{d}x}

\def\P{\mathbb{P}}

\newcommand{\norm}[1]{\left\| #1 \right\|}

\newcommand{\tn}{|\kern-.1em|\kern-0.1em|}

\newcommand{\vrc}[2]{\mathrm{Var}\left(#1\;\middle\vert\;#2\right)}

\newcommand{\red}[1]{{\color{red}{#1}}}

\newcommand\be{\begin{equation}}
\newcommand\ee{\end{equation}}

\def\eps{\varepsilon}

\newcommand{\tv}[1]{\left\|#1\right\|_{\rm{TV}}}

\newcommand{\tmix}[1]{t_{\mathrm{mix}}(#1)}

\newcommand{\tmixx}[2]{t_{\mathrm{mix}}(#1, #2)}

\newcommand{\td}[1]{\textbf{\red{[#1]}}}

\newcommand{\tb}[1]{\textbf{\color{blue}{#1}}}

\newcommand{\trunc}[2]{\mathrm{Tr}(#1,#2)}
\newcommand{\tree}{T}
\newcommand{\G}{G_n^*}

\newcommand{\quasi}{\text{quasi tree}}
\newcommand{\rball}[1]{#1-{$R$}-\text{ball}}

\begin{document}

\title{Universality of cutoff for quasi-random graphs}
\author{Jonathan Hermon
        \thanks{
                University of British Columbia, Vancouver, CA. E-mail: {jhermon@math.ubc.ca}. }
        \and Allan Sly
        \thanks{Princeton University,  Department of  Mathematics, Princeton, NJ, USA. E-mail: {asly@princeton.edu}} 
        \and Perla Sousi
        \thanks{Statistical Laboratory, Cambridge, UK. E-mail: {p.sousi@statslab.cam.ac.uk}} }
\date{}
\maketitle

\begin{abstract}
We establish universality of cutoff for simple random walk on a class of random graphs defined as follows. Given a finite graph $G=(V,E)$ with $|V|$ even we define a random graph $ G^*=(V,E \cup E')$ obtained by picking $E'$ to be the (unordered) pairs of a random perfect matching of $V$. We show that for a sequence of such graphs $G_n$ of diverging sizes and of uniformly bounded degree, if the minimal size of a connected component of $G_n$ is at least 3 for all $n$, then the random walk on $\G$ exhibits cutoff w.h.p. This provides a simple generic operation of adding some randomness to a given graph, which results in cutoff.
\newline
MSC 2010 \emph{subject classifications.} Primary 60J10, Secondary                05C80; 05C81.
.
\end{abstract}

\section{Introduction}

This paper is motivated by the question of what types of randomness one can add to a given family of graphs so that simple random walk on the resulting graph would exhibit cutoff. In this work we show that the operation of adding the edges of a random perfect matching leads to cutoff with high probability. 
More precisely, suppose that $G=(V,E)$ is a finite graph with $|V|$ even. We define a random graph $G^*=(V,E\cup E')$, where $E'$ is a uniformly random perfect matching of $V$. While this random graph  shares some features of some classical random graph models, such as the configuration model, it differs in that it retains some of the original structure $G$, and thus it has a richer local structure than many random graph models, which are locally tree-like. Diaconis in~\cite[Section 5, Question 4]{Persiwhatwelearned} posed the problem of determining the order of the mixing time in the case when~$G$ is connected and regular of constant degree and a perfect matching (random or deterministic) is added to $G$.    

Let $X$ be a simple random walk on a graph $G$ with transition matrix $P$ and invariant distribution~$\pi$. We define the $\eps$-\emph{total variation mixing time}
\[
\tmixx{G}{\epsilon} = \min\{ t\geq 0: \max_x\tv{P^t(x,\cdot) - \pi} \leq \epsilon\},
\]
where for $\mu$ and $\nu$ two distributions  we write $\tv{\mu-\nu}=\sum_{x}|\mu(x)-\nu(x)|/2$ for their total variation distance. For a sequence of graphs $(G_n)$, we say that the corresponding sequence of random walks exhibits \emph{cutoff} if
\begin{equation}
\label{e:cutoffdef}
\forall \, \eps \in (0,1), \qquad \lim_{n \to \infty} \frac{t_{\rm{mix}}(G_{n},\eps)}{t_{\rm{mix}}(G_{n},1/4)}=1.
\end{equation}
We say that an event $A$ happens with high probability (w.h.p.)\ if $\pr{A}=1-o(1)$ as $n\to\infty$.
When the graphs $G_n$ are random graphs, we say that cutoff holds w.h.p.\ if~\eqref{e:cutoffdef} holds in distribution.  Our main result is the following:
\begin{theorem}
\label{thm:1}
Let $G_n=(V_{n},E_{n})$ be a sequence of finite graphs of even diverging sizes of maximal degree at most $\Delta$, for some constant $\Delta \in \N$. Assume that the minimal size of a connected component of $G_n$ is at least $3$ for all $n$. Then the discrete time simple random walk on $\G$ exhibits cutoff w.h.p. Moreover, for all $\eps \in (0,1/2)$  there exists a  constant $C(\Delta,\eps)>0$ so that w.h.p.
\begin{equation}
\label{e:window}
\tmixx{\G}{\epsilon}-t_{\rm{mix}}(\G,1-\eps) \le C(\Delta,\eps) \sqrt{\log |V_n|}. 
\end{equation}
Finally w.h.p.\ $t_{\rm{mix}}(\G,1/4)\asymp \log |V_n| $. 
\end{theorem}

We recall that for an irreducible reversible Markov chain on a finite state space with transition matrix $P$ the \emph{absolute spectral gap} $\gamma$ is defined as
\[
\gamma= 1-\max\{|\lambda|: \lambda \text{ is an eigenvalue of } P \text{ with } \lambda\neq 1\}.
\]

\begin{proposition}\label{pro:expander}
In the setup of Theorem~\ref{thm:1} there exists $\alpha= \alpha(\Delta)>0$ such that if $\gamma_n$ denotes the absolute spectral gap of simple random walk on $\G$, then  w.h.p. 
\[
\gamma_n\geq \alpha.
\]
\end{proposition}

Proposition~\ref{pro:expander} immediately implies the last assertion of Theorem~\ref{thm:1} using the Poincar\'e inequality.  
 It turns out that the mixing time has an entropic description which is given  in terms of the random walk on some auxiliary infinite random graph which we refer to as the corresponding ``\emph{quasi tree}"~$\tree_n$ defined as follows (see Figure \ref{fig:quasitree}).\footnote{When the sequence of graphs $(G_n)_{n \in \N }$ has a Benjamini-Schramm limit $G$, one can define the entropic time using~$G$. This is discussed in Remark \ref{rem:BSlimit}.} Pick a random ball of $G_n$ of radius $R=R_{n}:=\lceil C \log \log |V_n| \rceil$. We refer to the centre of this ball as the \emph{root}. Each vertex $v$ in the ball, other than its centre, is connected by an edge to the centre of a random ball $B_{v}$ of radius $R$ (in $G_n$). The balls $B_v$ are picked independently. We refer to each such ball as an $R$\textbf{-ball}. Repeat this operation inductively, where at each stage  the centres of the balls do not have an edge emanating from them, and the rest of the vertices in each ball have a single edge emanating from them. Call the resulting graph $T_n$. The cutoff time is then given by the time at which the entropy of simple random walk on $T_n$ is $\log |V_n|$. The fluctuations of $\tmix{\epsilon}$ around this time are given by \eqref{e:window} up to a constant factor. Indeed Remark 1.6 in \cite{HLP} implies that the cutoff window is $\Omega(\sqrt{t_{\rm{mix}}(\G,1/4)/\Delta}) $. 

\begin{figure}[!h]
\begin{center}
\includegraphics[scale=0.6]{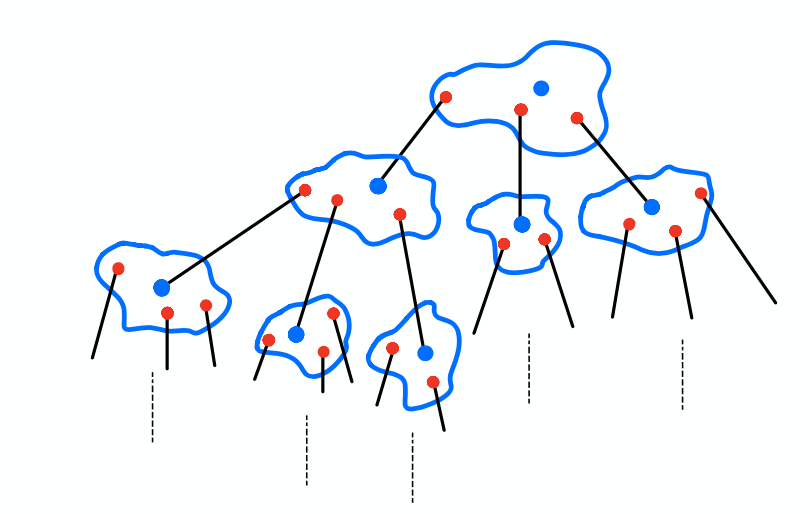}
\end{center}
\caption{\label{fig:quasitree} 
An illustration of a quasi tree. The blue vertices are the centers of the corresponding $R$-balls. The internal adjacency structure inside each ball is identical to the one in the corresponding ball in $G_n$.}
\end{figure} 
 
 \begin{remark}
        \emph{The assumption in Theorem~\ref{thm:1} that $|V_n|$ is even can be dropped by leaving one vertex unmatched when $|V_n|$ is odd. Our analysis can easily be extended to the graph obtained by ``super-positioning" a configuration model of bounded degree on $G_n$, obtained by adding to each vertex~$i$, $d_i \in [1,\Delta']$ half-edges (with $\sum_i d_i$ even) and then adding to $G_n$ the edges corresponding to a random perfect matching of the half-edges. }
 \end{remark}
 \begin{remark}
        \label{r:degree}
        \emph{An inspection of the proofs of Theorem~\ref{thm:1} and Proposition~\ref{pro:expander} reveals that for the lazy or the continuous-time versions of the walk, we can allow the maximal degree to be $O((\log n)^c)$ for some sufficiently small $c \in (0,1)$. In this case, the right hand side of~\eqref{e:window} would become larger, but would still remain $o(\log |V_n|)$.
        We note that some condition on the maximal degree is needed. The simplest example is obtained by taking $G_n$ to be the Cartesian product $K_{n/2} \times K_2$ where $K_m$ is a clique on $m$ vertices.  In fact, one can construct a family of examples of degree as small as $\Theta(\frac{\log n}{\log \log n})$. This construction is given in \S\ref{s:example}.}
 \end{remark}
Our lower bound on $t_{\rm{mix}}(\G,1-\eps)$, which as discussed above can be expressed using an entropic time, in fact holds w.h.p.\ simultaneously for all starting points (see Remark \ref{r:beststartingpoint}).

As we note below, cutoff for random walk on random graphs at an entropic time defined w.r.t.\ some auxiliary random walk is a  paradigm that emerged in the last few years. An interesting feature of our result is that the  random graph is not tree-like and the auxiliary random walk is not defined on a tree. The only other such cases that we are aware of in the literature are \cite{Cayley1,Cayley2}, which use completely different (group theoretic) methods.

 \subsection{Related work - cutoff at the entropic time for random instances paradigm}

We now put our results into a broader context.
A recurring theme in the study of the cutoff phenomenon is that random instances often exhibit cutoff. This was already observed by Aldous and Diaconis  in their seminal '86 paper \cite{AD:shuff-stop} where they coined the term cutoff.  
In this setup, a family of transition matrices chosen from a certain family of distributions is shown to give rise to a sequence of Markov chains which exhibits cutoff w.h.p. In recent years this has been verified for random walk on various natural random graphs. Lubetzky and Sly established the cutoff phenomenon for random walk on random regular graphs \cite{LS:cutoff-random-regular}. Together with Berestycki and Peres \cite{BLPS} they established cutoff for a typical starting point at an entropic time\footnote{With respect to a random walk on the corresponding Benjamini-Schramm limit, which is the size-biased version of a Poisson Galton-Watson tree. } for the random walk on the giant component of an Erd\H os-R\'enyi graph as well as on a random graph with a given degree sequence, satisfying some (very) mild assumptions on the degrees. Cutoff for the non-backtracking random walk on a random graph with a given degree sequence was established independently by Ben-Hamou and Salez~\cite{BhS:cutoff-nonbacktracking}. Ben-Hamou, Lubetzky and Peres verified cutoff at the same entropic time also for a worst-case initial point for the configuration model in \cite{BhLP:comparing-mixing}. Ben-Hamou in~\cite{Anna} also established cutoff for the non-backtracking random walk on a variant of the configuration model which incorporates a community structure.

A few other notable examples, where cutoff has been proved at an entropic time include random walks on a certain generalisation of Ramanujan graphs \cite{BL}. Cutoff for all Ramanujan graphs was proven earlier by Lubetzky and Peres in~\cite{LP:ramanujan}, and on random lifts \cite{BL,conchon2019cutoff} by Bordenave and Lacoin, and by Conchon--Kerjan. Two additional remarkable such examples, due to Bordenave, Caputo and Salez,  where the Markov chain is non-reversible and the stationary distribution is not well-understood, are random walks on random digraphs \cite{BCS:rw-sparse-digraphs} and a large family of sparse Markov chains \cite{BCS} obtained by permuting the entries of each row of the transition matrix independently. A similar model, which is even closer to our model, is studied by Chatterjee and Diaconis \cite{CD}. They showed that under mild assumptions on a doubly-stochastic transition matrix $P$, if $\Pi$ is a random permutation matrix, then $P \Pi$ w.h.p.\ has mixing time which is logarithmic in the size of the state space. See Bordenave et al.\ \cite{bordenave2018spectral} for a related work about the second largest eigenvalue in absolute value of  $\Pi P$. We note that while the last two examples bear some resemblance to our model, they differ in that the Markov chains there are locally tree-like. Moreover, the approach we employ to prove Proposition \ref{pro:expander}, involving a comparison with the configuration model, is different to the one in \cite{CD}, which is more combinatorial in nature.

Cutoff at an entropic time was recently established also for random walk on random Cayley graphs for all Abelian groups \cite{Cayley1} as well as for the group of unit upper triangular matrices with entries in $\mathbb{Z}_p$ \cite{Cayley2}.  Eberhard and Varj\'u \cite{eberhard2020mixing} established cutoff at an entropic time for the Chung--Diaconis--Graham random walk. Another recent application of ``the entropic method" for a problem related to repeated averages can be found in \cite{averages}.
Lastly, we mention that cutoff was established also for
        random birth and death chains \cite{DW:random-matrices,S:cutoff-birth-death}. It would be interesting to establish the same for a natural model of a random walk on a  random weighted tree.

A recurring idea in the aforementioned works  is that the cutoff time can be described in terms of entropy.
        One can look at some auxiliary random process which up to the cutoff time can be coupled with, or otherwise related to, the original Markov chain---often in the above examples this is the random walk on the corresponding Benjamini--Schramm local limit.
        The cutoff time is then shown to be (up to smaller order terms) the entropic time, defined as the time at which the entropy of the auxiliary process equals the entropy of the invariant distribution of the original Markov chain.

We finish this discussion with two very recent instances in which the entropic method was used to prove cutoff in setups where the Markov chain is non-random and the entropy is considered directly w.r.t.\ the chain, rather than some auxiliary ``limiting" chain. Ozawa \cite{ozawa} gave recently an entropic proof of the aforementioned result of Lubetzky and Peres \cite{LP:ramanujan} that random walks on Ramanujan graphs exhibit cutoff (see also \cite{BL,HerRam}). His proof gives a certain general condition in terms of concentration of $-\log P^{t}(x,\cdot)$ around the entropy of $P^t(x,\cdot)$ which implies cutoff for random walks on expanders. In a recent breakthrough, Salez \cite{Salez2021cutoff} develops a more general connection between such a concentration and cutoff involving the varentropy. His formulation is actually done in terms of relative entropy.  He then applies it to give sufficient conditions for cutoff for chains with non-negative curvature. In particular, he shows that random walks on expander Cayley graphs of Abelian groups exhibit cutoff.  
\subsection{Organisation}
In Section~\ref{s:overview} we give an overview of the ideas and techniques involved in the proof of Theorem \ref{thm:1}. In Section~\ref{s:speed} we define the notion of quasi trees and prove results concerning the speed and entropy of a random walk on them, as well as some concentration estimates around the entropy. In Section~\ref{s:coupling} we define a coupling of a portion of the random graph $G_n^*$ and a quasi tree and of the random walks on them. The coupling involves a certain truncation event defined and studied in Section~\ref{sec:truncation}. This coupling is then used to conclude the proof of Theorem~\ref{thm:1}. Finally in Section~\ref{s:expander} we prove Proposition~\ref{pro:expander}.

\section{Overview}
\label{s:overview}
Recall the construction of the {\quasi} $\tree$ that we described after the statement of Proposition~\ref{pro:expander}. Let $(X_t)$ be a random walk on it starting from its root. Let  $(X_t')$ be an independent copy of $(X_t)$, given $\tree$.  Loosely speaking, the first stage of the analysis is to show $-\frac{1}{t}\log\mathbb{P}(X_t=X_t' \mid X_t,T)$ converges as $t \to \infty$ to some value $\mathfrak{\hat h}$, and has variance  $O(t)$. Note that the randomness here is jointly over $\tree$ and the walk $(X_t)$. As we later explain, we first establish this for a certain notion of loop-erased walk, and then deduce a related statement for the random walk, whereas the above statement is never proven explicitly and is not used.  In the case that $\tree$ is a Galton-Watson tree, the convergence is classical \cite{ErgodicGW}, and this variance estimate is proven in~\cite{BLPS}.

We take an elementary approach to the problems of extending some of the known ergodic theory for random walks on Galton-Watson trees to the setup of quasi trees and of establishing the above variance estimate. Our approach involves exploiting a certain i.i.d.\ decomposition of the walk and the {\quasi}  (see Lemma \ref{lem:regenerationtimes}), using a natural analogue of the notion of regeneration times used to prove a similar decomposition for random walks on Galton-Watson trees (see the discussion before Lemma \ref{lem:regenerationtimes}). From a high-level perspective, our conceptual contribution here is two-fold: 
\begin{itemize}
\item[(i)] The observation that such a decomposition can be used also when $\tree$ is not a Galton-Watson tree, corresponding to the case that the random graph is not "tree-like".
\item[(ii)] The observation that such a decomposition is powerful enough to deduce concentration for  $-\frac{1}{t} \log \prcond{X_t=X_t'}{ X_t,\tree}{}$.  
\end{itemize}

The above concentration implies that if $t=(\log |V_n|-C_{\epsilon} \sqrt{\log |V_n|} )/\mathfrak{\hat h}$, for a suitable choice of $C_\epsilon$, then we can write the law of $X_t$ as $(1-\epsilon)\mu +\epsilon \nu$, where for all $x$ in the support of $\mu$
\[
\mu(x)\in [|V_n|^{-1}\exp(  C' \sqrt{\log |V_n|}),|V_n|^{-1}\exp( 2 C' \sqrt{\log |V_n|})] 
\] 
for a positive constant $C'$. If the same applies for the graph $\G$, then this shows that the random walk is not mixed at time $t$, as with probability at least $1-\eps$ it is supported on a set whose size is at most 
\[
|V_n|\exp(  -C' \sqrt{\log |V_n|})=o(|V_n|)
\]
To see this note that the support of a distribution $\mu'$ with $\min_{x:\, \mu'(x)>0} \mu'(x) \ge \delta$ has size at most $1/\delta$; use this with $\mu'=\mu$ and by the bounded degree assumption a set of size $o(|V_n|)$ has stationary measure $o(1)$. Moreover,  since we show (Proposition \ref{pro:expander}) that $\G $ is w.h.p.\ an expander, a standard application of the Poincar\'e inequality\footnote{Write the law of the walk at time $t$ as a mixture $ \eps \nu+(1-\eps)\mu$, with $\mu$ having $L_2$ distance at most $O(\exp(  C' \sqrt{\log |V_n|})) $ from the stationary distribution, and then apply the Poincar\' e inequality to $\mu$.} shows that this would imply that the random walk on $\G$ is well-mixed at time $t+\widetilde C \sqrt{\log |V_n|}$.

Motivated by the above, we shall couple a portion of the random graph  $\G$ rooted at a vertex $x$ with a portion of a {\quasi} in a certain manner that will facilitate a coupling of the random walks on these graphs up to the above time $t$. 
Several difficulties arise when implementing this approach. The first is that while the random graph $\G$ rooted at a vertex $x$ is typically (i.e., for most $x$)  locally indistinguishable from a {\quasi} from the perspective of the random walk, this fails for some $x \in V_{n}$. This turns out to not be a substantial obstacle. Following \cite{BhLP:comparing-mixing}, loosely speaking, we argue that w.h.p.\ $\G$ is such that for all starting points $x$ the walk is likely to reach a ``good" starting point for which the aforementioned coupling is successful with  probability close to 1. The good starting points will be ones that are locally {``\quasi}  like" in some precise sense.

 The second difficulty is that there is a limit to how one may hope to successfully couple a portion of the random graph $\G$ rooted at $x$ with a portion of a {\quasi}. Indeed, the $R$-balls in the {\quasi} are sampled at each stage at random with replacements, and in $\G$ without replacements. We attempt to couple the two graphs one ball at a time, using a maximal (i.e., optimal) coupling for the distribution of the balls. However, these maximal couplings may fail on some occasions, and they do so more often as the size of the portions of the two graphs we revealed exceeds $\sqrt{|V_n|}$, and becomes closer and closer to size $|V_n|$. When the maximal coupling fails, we may even get two $R$-balls in the portion of the random graph we revealed that overlap.

To overcome this difficulty, we argue that starting from a good vertex the random walk is unlikely to visit, by time $t$ defined above, any $R$-ball for which the coupling fails. Following~\cite{BLPS}, loosely speaking, we truncate both the {\quasi} and the portion of the random graph around $x$ which we reveal at edges with the property that the probability that the random walk crosses them by time $t$ is ``too small", say less than $|V_n|^{-1}\exp( \frac 12 C' \sqrt{\log |V_n|})$. This is crucial in avoiding revealing too many vertices, which would result in the coupling of the balls failing ``too often", while being able to couple the walks on the two graphs by time $t$ with a large success probability. The actual details of the argument vary slightly from this simplified description.

We now explain in more detail how we study the random walk on the \quasi. We refer to the edges connecting a vertex to a new $R$-ball as \emph{long range} edges. One can consider the induced walk on the long range edges, which is the walk viewed only at times when it crosses long range\ edges. One can then define the loop-erasure of this induced chain in a natural manner (see Definition \ref{def:lerw}). We say that a long range edge $e=(x,y)$ is a \emph{regeneration edge} if it is crossed, and after it is first crossed the random walk never returns to $x$. For a regeneration edge $e$, the time it is crossed is then called a \emph{regeneration time}. It is this notion which gives us the aforementioned decomposition of the walk and the {\quasi} into i.i.d.\ blocks (see Lemma \ref{lem:regenerationtimes} for a precise statement). Using this decomposition we derive the concentration of the analogue of (ii) above w.r.t.\ the loop-erasure. We then translate this into a corresponding claim concerning the random walk.

We use the fact that the connected components of $G_n$ are of size at least~$3$ to deduce that
\begin{itemize}
\item the walk on the {\quasi} has a positive speed, where distance is measured in the number of long range edges separating a point and the root of the \quasi, and
\item that the spacings between the regeneration times have an exponentially decaying tail.
\end{itemize}
 This plays a role both in deriving the aforementioned concentration estimate for the loop-erasure, as well as in translating it back to one concerning the random walk on the {\quasi}. For the sake of being precise, we note that we do not explicitly translate it exactly to the claim (ii) above, although this could be done without too much additional effort. We do not require this exact formulation, and thus do not pursue it.

We now provide an alternative description of the cutoff time. Let $(\xi_k)$ and  $(\xi_k')$ be independent (given $\tree$) loop-erased random walks on $\tree$ in the above sense started from its root.   We show that $-\frac{1}{k}\log\mathbb{P}(\xi_k=\xi_k' \mid \xi,T)$ converges a.s.\ to some constant $\mathfrak{h}$ as $k \to \infty$. We also show that the `speed' of the random walk $(X_t)$ on $\tree$, measured in the `long range distance' (the long range distance of~$x$ from the root is the level to which $x$ belongs) converges to some constant $\nu $. The cutoff time is then $\frac{\log n}{\nu \mathfrak{h}}$. We comment about the possibility of defining $\nu$ and $\mathfrak{h}$ in terms of a Benjamini-Schramm limit in Remark \ref{rem:BSlimit}.
The cutoff time resembles that in \cite{BLPS}. We note that this is a consequence of our definitions for loop-erased random walk and for speed, which are not the standard ones.

The assumption on the minimal size of a connected component of  $G_n$ is also used in bounding the spectral gap of $\G$. We essentially compare it to that of a random graph sampled from the configuration model with minimal degree at least 3 and bounded maximal degree. More effort is needed to bound the \emph{absolute} spectral gap.             

\emph{Notation:} For functions $f$ and $g$ we will write $f(n) \lesssim g(n)$ if there exists a constant $c > 0$ such that $f(n) \leq c g(n)$ for all $n$.  We write $f(n) \gtrsim g(n)$ if $g(n) \lesssim f(n)$.  Finally, we write $f(n) \asymp g(n)$ if both $f(n) \lesssim g(n)$ and $f(n) \gtrsim g(n)$.  Let $G = (V,E)$ be a graph and let $A\subseteq  V$. We write $\partial A$ for the internal vertex boundary of $A$, i.e.\ $\partial A = \{ x\in A: \exists \ y\notin A \text{ s.t. } \{ x,y\} \in E\}$.

\section{Speed and entropy of simple random walk on quasi trees}
\label{s:speed}

We start this section by recalling the construction of a quasi tree  $\tree=\tree_n$  from the introduction (see Figure \ref{fig:quasitree}). This will serve as an infinite approximation to the graph $\G$. Then we will prove scaling limit and fluctuation results for the entropy and the speed of simple random walk on $\tree$.

\begin{definition}\label{def:tree}
        \rm{
        Let $C_1>0$ be a constant. We define a (random) {\emph{\quasi}} $\tree=\tree_{C_1}$ to be an infinite graph constructed as follows. Let $B$ be a random ball (in the graph distance of $G_n$) obtained by first sampling a uniform vertex and then considering its $R=\lceil C_1\log \log n \rceil$ neighbourhood. We call such a ball a $\tree\rball$. 
        
        Let $\rho$ be its centre and we call it the root of $\tree$. Next join by an edge each other vertex $v$ of $B$ (except for the root) to the centre of an i.i.d.\ copy $B_v$ of $B$, i.e., the balls are sampled independently with replacement.  Repeat the same procedure for every vertex of the new balls except for their centres. We call edges joining different balls \emph{long range edges}. 
        
        The {\quasi} is a random variable taking values in the topological space $\mathcal{T}$ defined as the space of all rooted locally finite unlabelled connected graphs with a collection of distinguished edges, called long range edges, with the property that every simple path between a pair of vertices must cross the same collection of long range edges. In other words, the long range edges give rise to a tree structure.

For $x,y \in \tree$ we write $d_\tree(x,y)$ or simply $d(x,y)$ when $\tree$ is clear from context, for the number of long range edges on the shortest path from $x$ to $y$. Note that this is not the usual graph distance on $\tree$, but for us this will be a useful notion of distance. One can think of this distance as ``the long range distance", but since we rarely consider the graph distance on $T$, we do not use this terminology. A level consists of all vertices at the same distance from $\rho$, i.e.\ when $d(\rho,x)=r$, then~$x$ belongs to the $r$-th level. We write $\B_r(x)=\B(x,r)=\{y:d_\tree(x,y) \le r \}$ for the ball of radius $r$ centred at $x$. We also write $\tree(x)$ for the subgraph of $\tree$ rooted at $x$. More precisely, $\tree(x)$ is the induced graph on the vertices $y$ satisfying $d(\rho,y)=d(\rho,x)+ d(x,y)$. The vertices of $\tree(x)$ are called the \emph{descendants} or \emph{offspring} of $x$.
}
\end{definition}

\begin{remark}
\rm{We now explain the choice of $R$. We are going to define a coupling of the walk on the random graph with a walk on a quasi tree up to time $t$ of order $\log n$. In order for the coupling to succeed, we need to ensure that the walk on the {\quasi} does not reach the boundary of a $\tree\rball$ by time $t$. In order to achieve this we need to take $R$ of order at least $\log \log n$. The coupling also involves an exploration of a portion of the random graph at the same time with the corresponding {\quasi} (for both graphs we are primarily interested in the portion of the graph where the walk is likely to be by time $t$). As will become apparent, in order for this to succeed, we also need to ensure that by time $t$ we only reveal $o(n)$ vertices of $\G$ and that typically the other endpoint of long range edges we reveal satisfy that the balls of radius $R$ around them in $G_n$ are disjoint from the previously exposed such balls (as is the case for a quasi tree). This motivates us to take $R$ to be as small as possible.
We note that our results in this section about speed and entropy of random walk on $\quasi$ are not limited to this choice of $R$. For the sake of our results on speed and entropy of the walk we could have taken $R$ to be the diameter of the graph. In the case that the sequence $(G_n)_n$ has a Benjamini-Schramm limit, (for the aforementioned purposes) we could have taken the balls in the construction to be i.i.d.\ rooted copies of the Benjamini-Schramm limit. In fact, with a bit more care, one can derive from our analysis that up to subleading order terms the speed of the walk and its entropy would be the same in these cases, as in our construction, and similar concentration bounds hold also in these cases.
Taking $R$ to be the diameter or using the Benjamini-Schramm limit (when it exists) may seem more natural, at least from the perspective that results about the speed and entropy of the walk in these cases are of interest in their own right. However, as will become clear, for the sake of proving cutoff our choice of $R$ is natural.}
\end{remark}

For a Markov chain $X$ and a vertex $x$ we denote the first hitting time of~$x$ by $\tau_x = \inf\{t\geq 0:X_t=x\}$ and by 
$\tau_x^+=\inf\{t\geq 1: X_t=x\}$ the first return time to $x$. 

\begin{lemma}\label{lem:uniformdrift}
Let $\tree$ be a {\quasi} as in Definition~\ref{def:tree}. Let $X$ be a simple random walk on $\tree$. For every $x\in \tree$ which is not in the $\tree\rball$ of the root, we write $p(x)$ for the ``parent'' of the centre of the $\tree\rball$ containing $x$, i.e.\ $p(x)$ satisfies $d(\rho,p(x)) = d(\rho,x)-1$. For $x$ in the  $\tree\rball$ of the root, we set $p(x)=\rho$.  Let $\mathbb{P}_x$ denote the law of the random walk on $T$ started from~$x$.
 Then there exists a positive constant $c$ so that for all $n$ and for every realisation of $\tree$ 
\[
\prstart{\tau_{p{(x)}}^{+}\wedge \tau_x^+=\infty}{x} \geq c, \text{ for all }  x\in \tree.
\]
%
\end{lemma}

\begin{proof}[\bf Proof]
        It will be useful in the proof to think of vertices of $\tree$ lying in half and full levels as follows. All centres at the same distance from the root are placed in the same half level. Their neighbours in the corresponding balls are placed in the same full level. We now change the definition of distance to take into account half levels, i.e.\ the distance between a centre and other points in its ball is equal to~$\tfrac12$ and the distance between two endpoints of a long range edge is also $\tfrac12$. We denote this distance by $\til{d}$ and it satisfies $\til{d}(\rho,x)=d(\rho,x)-\1(x \text{ is a centre of a $\tree\rball$})/2$.
        We next claim that for all $x$ we have 
        \begin{align}\label{eq:positivedriftalways}
                \econd{\til{d}(\rho, X_1)-\til{d}(\rho,x)}{X_0=x} \geq 0.
        \end{align}
Suppose first that $x$ is a vertex which is neither a centre nor a neighbour of a centre. Then 
        \begin{align}\label{eq:strictlypos}
        \econd{\til{d}(\rho, X_1)-\til{d}(\rho,x)}{X_0=x} \geq \frac{1}{2(\Delta+1)},
        \end{align}
        i.e.\ there is positive drift downwards. If $x$ is a centre with at least two neighbours in its corresponding ball, then 
        \begin{align}\label{eq:strictlyposit}
        \econd{\til{d}(\rho, X_1)-\til{d}(\rho,x)}{X_0=x} \geq \frac{1}{6}.
        \end{align}
        Finally, suppose that $x$ is either a centre of degree equal to $2$ or $x$ is a neighbour of a centre. In both cases we have 
        \[
        \econd{\til{d}(\rho, X_1)-\til{d}(\rho,x)}{X_0=x} =0.
        \]
        This concludes the proof of~\eqref{eq:positivedriftalways}.
        We now look at the distance from the root at times that are multiples of $3$, i.e.\ we consider $Y_t=\til{d}(\rho, X_{3t})$ and write $\F_t=\sigma(X_i,i\leq 3t)$. We next show that there exists a positive constant $\delta$ such that 
        \begin{align}\label{eq:boundonincrements}
        \econd{Y_{t+1}-Y_t}{\F_t}\geq \delta >0.
        \end{align}
        We start by writing the conditional expectation above as follows
           \begin{align*}
        &\econd{Y_{t+1}-Y_t}{\F_t} = \econd{\til{d}(\rho,X_{3t+3})-\til{d}(\rho,X_{3t+2})}{\F_t} \\&+ \econd{\til{d}(\rho,X_{3t+2})-\til{d}(\rho,X_{3t+1})}{\F_t} 
        + \econd{\til{d}(\rho,X_{3t+1})-\til{d}(\rho,X_{3t})}{\F_t}. 
        \end{align*}
        It then follows from~\eqref{eq:positivedriftalways} that all terms appearing on the right hand side above are always non-negative. 
       We now consider different cases for $X_{3t}$ in order to show that at least one of the three terms in the r.h.s.\ is  strictly positive. 
       Let $K$ be the set of  vertices of $\tree$ that are centres and let $K_2$ be the subset of $K$ consisting of  those  centres which have at least two neighbours in their corresponding balls. Let also $N$ be the set of vertices of $\tree$ that are neighbours of centres. We write $A=(K^c \cap N^c) \cup K_2$. Then on the event $\{X_{3t} \in A\}$ we have 
              \begin{align*}
       \econd{\til{d}(\rho,X_{3t+1})-\til{d}(\rho,X_{3t})}{\F_t}  \geq  \min\left(\frac{1}{6} , \frac{1}{2(\Delta+1)}\right).
       \end{align*}
       On the event $\{X_{3t} \in N\}$ we get that there exists $x\in A$ with $P(X_{3t},x)\geq 1/(\Delta+1)$, where $P$ stands for the transition matrix of $X$ (indeed, either the centre of the ball to which $X_{3t}$ belongs is in $K_2$, or $X_{3t}$ has a neighbour in the same ball which is not adjacent to the centre of the ball).  So, writing $p_1=\min\left(\frac{1}{6(\Delta+1)} , \frac{1}{2(\Delta+1)^2}\right)$ on the event $\{X_{3t} \in N\}$ we get 
       \begin{align*}
      & \econd{\til{d}(\rho,X_{3t+2})-\til{d}(\rho,X_{3t+1})}{\F_t} \\& = \sum_{z} P(X_{3t},z)\econd{\til{d}(\rho,X_{1})-\til{d}(\rho,z)}{X_0=z} \geq  p_{1},
       \end{align*}
       where we also used again~\eqref{eq:positivedriftalways}. On the event $\{X_{3t}\in K\setminus K_2\}$ we get that there exists $x\in A$ with $P^2(X_{3t},x) \geq 1/(\Delta+1)^2$, and hence this gives that on $\{X_{3t}\in K\setminus K_2\}$ 
        \begin{align*}
       &\econd{\til{d}(\rho,X_{3t+3})-\til{d}(\rho,X_{3t+2})}{\F_t} \\& = \sum_{z} P^2(X_{3t},z)\econd{\til{d}(\rho,X_{1})-\til{d}(\rho,z)}{X_0=z} \geq p_2 ,
       \end{align*}
 where $p_2=\min\left(\frac{1}{6(\Delta+1)^2} , \frac{1}{2(\Delta+1)^3}\right)$.       This concludes the proof of~\eqref{eq:boundonincrements}.
Let $t \in \N$ to be determined.        Consider now the Doob martingale 
\[
M_{\ell}=Y_{t+\ell} - Y_t -\sum_{i=1}^{\ell} \econd{Y_{t+i}-Y_{t+i-1}}{\F_{i-1}}. 
\]
This has bounded increments, as the distance can only change by at most $3/2$ in $3$ steps of the walk. 

There exist positive constants $c_{1} ,c_{2}\in (0,1)$ so that 
\[
\mathbb{P}(D \mid X_0=x)\geq c_1, \quad \text{where }D:=\left\{ Y_s \geq Y_0 + \frac{3}{2} \text{ for all } s\leq t, \ Y_t\geq Y_0+c_2t+\frac{3}{2} \right\},
\]
where $c_1$ may depend on our choice of $t$. Let $r=\frac{2c_2t}{3}$.  By Azuma-Hoeffding and using~\eqref{eq:boundonincrements} 
\begin{equation}
\label{e:supermartingaleY}\begin{split}
&\prcond{Y_{t+\ell}\leq Y_t-c_2t }{X_0=x,D}{}  \\ &\le \prcond{M_{\ell} \leq -c_2t-\delta \ell }{X_0=x,D}{} \1\{\ell \ge r \}    \leq e^{-c \ell   }\1\{\ell \ge r \},
\end{split}
\end{equation}
for some constant $c>0$. Therefore,
 taking $t$ large enough and summing over $\ell$ we get that there exists a positive constant $c_3$ so that 
\begin{align}\label{eq:boundfory'}
\prcond{Y_s > Y_0+\frac{3}{2} \text{ for all } s\geq t}{X_0=x,D}{} \geq c_3.
\end{align}
Therefore, 
\begin{align*}
 \prstart{\tau_{p{(x)}}\wedge \tau_x^+=\infty}{x}        \geq c_1 c_3>0,
\end{align*}
 and this concludes the proof.
\end{proof}

\begin{remark}\label{rem:decayofdistance}
        \rm{
        Note that the above proof also gives that there exist positive constants $c_1$ and $c_2$ so that for all $t$
        \[
        \pr{d(\rho,X_{t})\leq c_1 t}\leq e^{-c_2t}.
        \]
        }
\end{remark}

\begin{definition}
        \rm{Let $\tree$ be a {\quasi} as in Definition~\ref{def:tree} and let $X$ be a simple random walk on~$\tree$. 
 A random time $\sigma$ is called a \emph{regeneration time} for $X$ if the long range edge $\{X_{\sigma-1}, X_\sigma\}$ is crossed for the first and last time at time $\sigma$. (We use $\{a,b \}$ to denote an undirected edge connecting $a$ and $b$, whereas $(a,b)$ to denote a directed edge from $a$ to $b$.)}
\end{definition}

Using Lemma~\ref{lem:uniformdrift} together with Remark \ref{rem:decayofdistance} we get 
that there are infinitely many  regeneration times almost surely.

The authors of \cite{BLPS} attribute to Kesten the ``tree analogue" (i.e.\ the case where $T$ is taken to be a Galton-Watson tree)  of the following lemma. The tree analogue was reproduced in \cite{MR1439974}. A similar statement is proved in~\cite[Proposition~3.4]{LyonsPemantlePeres} and our proof is similar to theirs. We include the proof here for the sake of completeness. Recall that $d$ denotes the ``long range" distance, and not the graph distance.

\begin{lemma}\label{lem:regenerationtimes}
Let $\tree$ be a {\quasi} as in Definition~\ref{def:tree} with root $\rho$. Fix $K\geq 0$ and let $T_0$ be a realisation of the first $K$ levels of $\tree$. 
Let $X$ be a simple random walk on $\tree$ started from the root. Let $\tree^a$ be the graph obtained by joining the root of $\tree$ to a new vertex $\rho^a$ by a single edge and let $\til{X}$ be a simple random walk on $\tree^a$ started from $\rho$. Let $\sigma_0$ be the first time that $X$ reaches $\partial \B_K(\rho)$.  Let $\sigma_i$ be the $i$-th regeneration time satisfying $\phi_i=d(\rho,X_{\sigma_i}) > K$ (i.e.\ $(\sigma_i)_{i=1}^{\infty}$ are the regeneration times after the last visit to $\B_K(\rho)$).
Then conditional on $\B(\rho,K)=T_0$,
 we have that 
 \begin{itemize}
        \item $(\tree(X_{\sigma_i})\setminus \tree(X_{\sigma_{i+1}}), (X_t)_{\sigma_{i}\leq t\leq \sigma_{i+1}})$ are i.i.d.\ for $i \ge 1$, and are jointly independent of  $(\tree\setminus \tree(X_{\sigma_{1}}), (X_t)_{0\leq t\leq \sigma_{1}})$,
        \item $(\sigma_i-\sigma_{i-1})_{i\geq 1}$ and $(\phi_i-\phi_{i-1})_{i\geq 1}$ have exponential tails and
        \item for all $i\geq 1$, the pair $(\tree(X_{\sigma_i}), (X_t)_{t\geq \sigma_i})$ has the law of $(\tree,\til{X})$ given that $\til{X}$  never visit $\rho^a$. (Note that this conditioning also affects the law of $\tree$.)
 \end{itemize}
\end{lemma}
We emphasise that above we view $\tree$ and $\tree(X_{\sigma_i})$ as rooted graphs defined up to graphs isomorphisms which preserve the root.
\begin{remark}
\rm{The conditioning on $\B(\rho,K)=T_0$ is not needed either for the i.i.d.\ decomposition or for deriving the later results about the speed, entropy and concentration around the entropy for the walk. However, the fact that such results hold even under the conditioning on   $\B(\rho,K)=T_0$, will be useful for the cutoff analysis later on.}
\end{remark}

\begin{proof}[\bf Proof]
        
        Following \cite{ErgodicGW} we define the set ($\mathrm{Q}$ below stands for ``quasi")
        \begin{align}\label{eq:pathsinqtrees}
        \begin{split}
        \mathrm{PathsInQTrees}=\{(\tree, (x_i)_{i \ge 0}):\tree\in \mathcal{T}, \, (x_i)_{i \ge 0} \text{ is a path in $\tree$} \\ \text{starting from its root}  \},
        \end{split}
        \end{align}
       where we recall that $\mathcal{T}$ was defined in Definition~\ref{def:tree}.
        We equip the space $\mathrm{PathsInQTrees}$  with the $\sigma$-algebra generated by $(\tree,X)$, where $\tree$ is the random quasi tree from Definition \ref{def:tree} and $X=(X_i)_{i \ge 0}$ is a simple random walk on $\tree$ started from its root.

Using Lemma~\ref{lem:uniformdrift} together with Remark \ref{rem:decayofdistance} we get the existence of the infinite sequence of regeneration times with the property that $(\sigma_i-\sigma_{i-1})_{i\geq 1}$ and $(\phi_i-\phi_{i-1})_{i\geq 1}$ have exponential tails.
Analogously to~\eqref{eq:pathsinqtrees} we define
\begin{align*}
\mathsf{PathsInAugQTrees}=\{(\tree, (x_i)_{i \ge 0}):\tree \in \mathcal{T}, \, (x_i)_{i \ge 0} \text{ is a path in} \\ \text{$\tree^{a}$ starting from its root}  \} 
\end{align*}
 and equip it with the $\sigma$-algebra generated by $(\tree, X)$, where $\tree$ is a (random) quasi tree, and $X$ is simple random walk on $\tree^a$ started from its root.
For a set~$A \subset \mathsf{PathsInAugQTrees}$
 we write 
\begin{align*}
        Q(A) = \prstart{(T,X)\in A, \tau_{\rho^a}=\infty}{\rho},
\end{align*} 
where $X$ is a simple random walk on $T^a$ started from $\rho$.  For a vertex $v$ which is a centre of some $\tree \rball$ we write $\tree_v$ for the tree obtained by removing from $\tree$ all of $\tree(v)$ other than $v$ itself ($\tree_v$ has the same root as $\tree$). In order to prove the i.i.d.\ property, it suffices to show that  conditional on $\B_K(\rho)=T_0$, for all $i \ge 1$ we have that $(\tree(X_{\sigma_i}), (X_k)_{k\geq \sigma_i})$ is independent of $(\tree_{X_{\sigma_i}}, (X_{k})_{k\leq \sigma_i})$  and to verify the stationarity of  $(\tree(X_{\sigma_i}), (X_t)_{t\geq \sigma_i})_{i \ge1}$. The stationarity will follow from the proof of independence. Let $A\subset \mathsf{PathsInAugQTrees}$ and $B\subset \mathsf{PathsInQTrees}$. To simplify notation we write $\prstart{\cdot}{T_0}$ for the probability measure $\prcond{\cdot}{\B_K(\rho)=T_0}{}$. We then have 
\begin{align*}
&\prstart{(\tree(X_{\sigma_i}), (X_k)_{k\geq \sigma_i})\in A, (\tree_{X_{\sigma_i}}, (X_{k})_{k\leq \sigma_i})\in B}{T_0} \\
        &= \sum_t \prstart{\sigma_i=t,(\tree(X_{t}), (X_k)_{k\geq t})\in A\cap \{\tau_{X_{t-1}}^t=\infty\}, (\tree_{X_{t}},(X_k)_{k\leq t})\in B }{T_0},
\end{align*}
where $\tau^t_y=\inf \{\ell \ge t : X_{\ell}=y \}$ denotes the first hitting time of $y$ by the chain $(X_k)_{k\geq t}$ and we treat $X_{t-1}$ as the new vertex ($X_t^a$ in the above notation) we attach to $\tree(X_t)$.
We say that a time $t$ is fresh if the walk visits $X_t$ for the first time at time $t$. Let $A_{k,t}$ be the event that there are exactly~$k$ regeneration times before $t$ when we only consider the walk up to time $t$. (By this we mean that the notion of being a regeneration time is now defined with respect to the length~$t$ walk.)
Then we have
\begin{align*}
        &\prstart{\sigma_i=t,(\tree(X_{t}), (X_k)_{k\geq t})\in A\cap \{\tau_{X_{t-1}}^t=\infty\}, (\tree_{X_{t}}, (X_k)_{k\leq t})\in B }{T_0}{}\\ =& \prstart{t \text{ fresh}, A_{i-1,t}, (\tree(X_{t}), (X_k)_{k\geq t})\in A\cap \{\tau_{X_{t-1}}^t=\infty\}, (\tree_{X_{t}},(X_k)_{k\leq t})\in B }{T_0}{}\\
        = &\prcond{(\tree(X_{t}), (X_k)_{k\geq t})\in A\cap \{\tau_{X_{t-1}}^t=\infty\}}{t \text{ fresh}, A_{i-1,t}, (\tree_{X_{t}}, (X_k)_{k\leq t})\in B}{T_0}\\ &\times \prstart{t \text{ fresh}, A_{i-1,t}, (\tree_{X_{t}}, (X_k)_{k\leq t})\in B}{T_0}{} \\
        =&Q(A)  \prstart{t \text{ fresh}, A_{i-1,t}, (\tree_{X_{t}}, (X_k)_{k\leq t})\in B}{T_0}{}.
\end{align*}
Taking now the sum over all times $t$ of the last probability above gives
\begin{align*}
       & \sum_t \prstart{t \text{ fresh}, A_{i-1,t}, (\tree_{X_{t}},(X_k)_{k\leq t})\in B}{T_0}{} \\&=\sum_t\frac{\prstart{\sigma_i=t, (\tree_{X_{t}}, (X_k)_{k\leq t})\in B}{T_0}{}  }{Q(\tau_{\rho^a}=\infty)} =\frac{\prstart{(\tree_{X_{\sigma_i}},(X_k)_{k\leq \sigma_i})\in B}{T_0}{} }{Q(\tau_{\rho^a}=\infty)}.
\end{align*}
(We note that $Q(\tau_{\rho^a}=\infty)=Q(\mathsf{PathsInAugQTrees})=\prstart{ \tau_{\rho^a}=\infty}{\rho}$.)
Therefore, putting everything together gives 
\begin{align*}
&\prstart{(\tree(X_{\sigma_i}), (X_k)_{k\geq \sigma_i})\in A, (\tree_{X_{\sigma_i}}, (X_k)_{k\leq \sigma_i})\in B}{T_0}{}\\&= \frac{Q(A)}{Q(\tau_{\rho^a}=\infty)}\cdot \prstart{(\tree_{X_{\sigma_i}}, (X_k)_{k\leq \sigma_i})\in B}{T_0}{}, 
\end{align*}
and hence this proves the claimed independence. Taking $B$ to be the whole space also proves the claimed stationarity of $(\tree(X_{\sigma_i}), (X_t)_{t\geq \sigma_i})_{i \ge1}$  and confirms the description of the law of $(\tree(X_{\sigma_i}), (X_t)_{t\geq \sigma_i})$ for $i \ge 1$ described in the last sentence in the statement of the lemma. Using similar reasoning one can verify that $(\tree(X_{\sigma_1}), (X_t)_{t\geq \sigma_1})$ and $(T \setminus\tree(X_{\sigma_1}), (X_t)_{0 \le t \le \sigma_1})$  are independent (proof omitted). This completes the proof.
\end{proof}

\begin{remark}\label{rem:everyreal}
\rm{
We note that from the proof of Lemma~\ref{lem:regenerationtimes} we see that for every realisation $\mathfrak t$ of $\tree$ we have that $(\sigma_i-\sigma_{i-1})_{i\geq 1}$ and $(\phi_i-\phi_{i-1})_{i\geq 1}$ have exponential tails.
}       
\end{remark}

\begin{definition}
        \rm{
        As in Lemma~\ref{lem:regenerationtimes}, we write $\sigma_0$ for the first time that $X$ reaches $\partial \B_K(\rho)$, $(\sigma_i)_{i\geq 1}$ for the sequence of regeneration times of $X$ occurring after time $\sigma_0$ and $\phi_i$ for the depth of~$X_{\sigma_i}$ for each $i$, when we condition on the event $\B_K(\rho)=T_0$. 
        }
\end{definition}

\begin{claim}\label{cl:renewal}
Let $\tree$ be a {\quasi} with root $\rho$ as in Definition~\ref{def:tree}. Fix $K\geq 0$ and let $T_0$ be a realisation of the first $K$ levels of $\tree$. 
For each $k\in \N$ let 
\[
N_k = \max\{ i\geq 0: \phi_i\leq k+K\}
\]
be the number of regeneration times occurring before level $k+K+1$. (As always, regeneration times are defined after time $\sigma_0-1$.) Then almost surely 
\[
\frac{N_k}{k} \to\frac{1}{\E{\phi_2-\phi_1}} \ \text{ as } \ k\to\infty.
\]
Moreover, for all $\epsilon>0$, there exists $C$ sufficiently large such that for all $k\geq K^2$
        \[
        \prcond{\left| N_k - \frac{k}{\E{\phi_2-\phi_1}} \right|>C\sqrt{k}}{\B_K(\rho)=T_0}{} \leq \epsilon.
        \]
\end{claim}

\begin{proof}[\bf Proof]
The almost sure convergence follows directly from the renewal theorem together with Lemma~\ref{lem:regenerationtimes}.

For the second statement, we only prove one bound. The other one follows in exactly the same way. Let
\[
\ell = \left\lfloor\frac{k}{\E{\phi_2-\phi_1}} + C\sqrt{k}\right\rfloor,
\]
where $C$ is a constant to be determined later. 
Set $\zeta_i=\phi_i - \phi_{i-1}$ for $i\geq 2$ and $\zeta_1=\phi_1$. We then have 
\begin{align*}
        \prcond{N_k > \ell}{\B_K(\rho)=T_0}{} = \prcond{\sum_{i=1}^{\ell} \zeta_i <k+K}{\B_K(\rho)=T_0}{}
\\ \leq  \prcond{\sum_{i=1}^{\ell} \zeta_i - \E{\sum_{i=1}^{\ell} \zeta_i} <2 \E{\zeta_2} - \E{\zeta_1} -C\E{\zeta_2}\sqrt{k}+K}{\B_K(\rho)=T_0}{}.
\\ \leq  \prcond{\sum_{i=1}^{\ell} \zeta_i - \E{\sum_{i=1}^{\ell} \zeta_i} <2 \E{\zeta_2}  -C\E{\zeta_2}\sqrt{k}}{\B_K(\rho)=T_0}{}.
\end{align*}

Since by Lemma~\ref{lem:regenerationtimes} $\zeta_2$ and $\zeta_1$ have exponential tails (in fact, since $K^2 \lesssim k$, for $\zeta_1$ it suffices to use below the bound $\Var(\zeta_1) \lesssim K^2$) and $(\zeta_i)_{i \ge 2}$ are i.i.d.\ and independent of $\zeta_1$,  using Chebyshev's inequality this last probability can be bounded by
\[
\prcond{\left|\sum_{i=1}^{\ell} \zeta_i - \E{\sum_{i=1}^{\ell}\zeta_i}\right| > C'\sqrt{k}}{\B_K(\rho)=T_0}{} 
\leq \frac{\Var(\sum_{i=1}^{\ell}\zeta_i)}{(C')^2k} \asymp\frac{\ell}{(C')^2k},
\]
for a positive constant $C'$, where the last equivalence follows again from Lemma~\ref{lem:regenerationtimes}. Taking $C$ large enough which implies that $C'$ is large, this last probability can be made smaller than $\epsilon$.
\end{proof}

\begin{lemma}\label{lem:speed}
        Let $X$ be a simple random walk on $\tree$. Then for $\nu=\frac{\E {\phi_2-\phi_1}}{\E {\sigma_2-\sigma_1}}$  almost surely
        \[
\frac{  d_\tree(\rho, X_t)}{t}\to \nu \text{ as } t\to\infty.
        \]
        Moreover, for all $\epsilon>0$ there exists a positive constant $C$ so that for all $t$ sufficiently large
        \[
        \pr{|d_\tree(\rho,X_t)-\nu t|>C\sqrt{t}}\leq \epsilon \quad \text{ and } \quad \pr{\sup_{s:\,s\leq t}d_\tree(\rho,X_s)>\nu t+ 2C\sqrt{t}}\leq \epsilon.
        \]
\end{lemma}

\begin{proof}[\bf Proof]
        The first and second claims follow easily using the regeneration structure from Lemma~\ref{lem:regenerationtimes} together with Claim~\ref{cl:renewal}. 
        
                For the final claim, let $C$ be such that the first inequality holds. Then we have 
        \begin{align*}
        &\pr{\sup_{s:\,s\leq t}d_\tree(\rho,X_s)>\nu t+ 2C\sqrt{t}} \\&\leq \pr{d_\tree(\rho,X_t)<\nu t+C\sqrt{t}, \ \sup_{s:\,s\leq t}d_\tree(\rho,X_s)>\nu t+ 2C\sqrt{t}}       +\epsilon\\
        &\leq \sum_{s:\, s\leq t} \pr{d_\tree(\rho,X_s)>\nu t+ 2C\sqrt{t},\ d_\tree(\rho,X_t)<\nu t+ C\sqrt{t}}+\epsilon\\
        &\leq t \cdot\epsilon e^{-c\sqrt{t}} + \epsilon,
                \end{align*}
        where $c$ is a positive constant and where for the final inequality we used that by Lemma~\ref{lem:uniformdrift} the probability that the walk goes up $i$ levels decays exponentially in $i$. 
\end{proof}

\begin{definition}\label{def:lerw}
\rm{
Let $\tree$ be a {\quasi} as in Definition~\ref{def:tree}. 
A \textbf{loop erased random walk} $\xi$ on $\tree$ is defined as follows: we run a simple random walk on $\tree$ for infinite time and we erase loops in the chronological order in which they are created. Usually one calls the obtained random simple path the loop erased random walk, however we employ the following different convention: for each $i$ we define $\xi_i$ to be the $i$-th long range edge crossed by this loop erasure. Unless otherwise specified, the loop erased walk $\xi$ is considered with respect to a walk started from the root of $\tree$.
}       
\end{definition}

The following lemma is a direct consequence of the domain Markov property for the loop erased walk $\xi$. We state it separately, since we will refer to it several times in the following proofs.

\begin{lemma}\label{lem:domainmarkov}
Let $\tree$ be a {\quasi} and let $T_0$ be its first $M$ levels for some $M>0$.  Let
        $X$ be a simple random walk on $\tree$ (resp.\ killed when exiting~$T_0$) and  let $\xi$ be its loop erasure as in Definition~\ref{def:lerw}. Let $(e_i=(x_i,y_i))_{i \in \N}$ be long range edges satisfying $d(\rho,x_i)<d(\rho,y_i)$ and $x_{i+1}$ is in the $\tree\rball$ centred at $y_i$ for all $i$. Then for every realisation of $\tree$, setting $\gamma=\{e_1,\ldots,e_k\}$ we have for all $k$ that 
        \[
        \prcond{\xi_{k+1}=e_{k+1}}{(\xi_i)_{i\leq k}=\gamma}{} = \prstart{\til{\xi}_1=e_{k+1}}{y_{k}}=\prstart{\til{X}_L=x_{k+1}}{y_k},
        \]
        where $\til{X}$ is a simple random walk on $\tree(y_k)$ (resp.\ on $\tree(y_k)\cap T_0$) started from $y_k$ whose loop erasure is $\til{\xi}$ and $L$ is the last time (resp.\ before exiting $T_0$) that $\til{X}$ is in the $\tree\rball$ centred at $y_{k}$. 
\end{lemma}

\begin{proof}[\bf Proof]

The lemma follows directly from the domain Markov property of loop erased random walk together with the ``tree-like'' structure of the {\quasi} $\tree$. 
\end{proof}

\begin{lemma}\label{lem:iidentropy}
        There exist positive constants $(C_\ell)_{\ell\geq 1}$ and $C'$ so that the following hold:
        let $\tree$ be a {\quasi} with root $\rho$ as in Definition~\ref{def:tree}. Fix $K\geq 0$ and let $T_0$ be a realisation of the first $K$ levels of $\tree$. 
        Let $X$ be a simple random walk on $\tree$ started from $\rho$ and let~$\til{\xi}$ be an independent loop erased random walk on $\tree$. For $k\geq 1$ define
        \[
        Y_k=-\log \prcond{(X_{\sigma_k-1},X_{\sigma_k})\in \til{\xi}}{X,\tree}{} + \log \prcond{(X_{\sigma_{k-1}-1}, X_{\sigma_{k-1}})\in \til{\xi}}{X,\tree}{}. 
        \]
Then the sequence $(Y_k)_{k\geq 2}$ is stationary and independent of $\B_K(\rho)$. Moreover, for all~$\ell\geq 1$ 
        \begin{align}\label{eq:tau0bound}
        \econd{\left( -\log \prcond{(X_{\sigma_0-1},X_{\sigma_0})\in \til{\xi}}{X,T}{} \right)^\ell}{\B_K(\rho)=T_0}\leq C_\ell (RK)^\ell
        \end{align}
        and for all $k\geq 2$
        \begin{align}\label{eq:momentsofy}
        \econd{(Y_k)^\ell}{\B_K(\rho)=T_0}{} \leq C_\ell \quad \text{and} \quad \E{|Y_1|^\ell}\leq C_\ell.
        \end{align}
        In addition, there exists a positive constant $C'$ so that for all $k\geq 1$ we have 
        \begin{align}\label{eq:varofsumofy}
        \vrc{\sum_{i=1}^{k} Y_i}{\B_K(\rho)=T_0} \leq C'k. 
        \end{align}
\end{lemma}

\begin{proof}[\bf Proof]
  To simplify notation we identify $X_{\sigma}$ with the long range edge $(X_{\sigma-1}, X_\sigma)$. 
Recall that the regeneration times were defined to be the times when a long range edge is crossed for the first and last time. This definition together with the fact that $\til{\xi}$ is only considered when the loop erasure crosses long range edges give that if for some $k\geq 2$ we have $X_{\sigma_k}\in \til{\xi}$, then also $X_{\sigma_{k-1}}\in \til{\xi}$. Using this and recalling that $\phi_k$ is the depth of $X_{\sigma_k}$ we obtain
\begin{align*}
        \prcond{X_{\sigma_k}\in \til{\xi}}{X,\tree}{} &=        \prcond{X_{\sigma_k}\in \til{\xi}, X_{\sigma_{k-1}}\in \til{\xi}}{X,\tree}{} \\
        &= \econd{\prcond{X_{\sigma_k}\in \til{\xi}, X_{\sigma_{k-1}}\in \til{\xi}}{X, \tree, (\til{\xi}_\ell)_{\ell\leq \phi_{k-1}}}{}}{X,\tree}\\
        &=\econd{\1(X_{\sigma_{k-1}}\in\til{\xi}) \prcond{X_{\sigma_k}\in \til{\xi}}{X, \tree, (\til{\xi}_\ell)_{\ell\leq \phi_{k-1}}}{}  }{X,\tree}.
\end{align*}
Using Lemma~\ref{lem:domainmarkov} we obtain
\begin{align*}
        \1(X_{\sigma_{k-1}}=\til{\xi}_{\phi_{k-1}}) &\prcond{X_{\sigma_k}\in \til{\xi}}{X, \tree, (\til{\xi}_\ell)_{\ell\leq \phi_{k-1}}}{}  \\= &\1(X_{\sigma_{k-1}}=\til{\xi}_{\phi_{k-1}}) \prcond{X_{\sigma_k}\in \xi(k)}{X,\tree}{},
\end{align*}
where $\xi(k)=(\xi(k)_i)_{i\geq 0}$ is a loop erased random walk on the subgraph $\tree(X_{\sigma_{k-1}})$ started from its root, $X_{\sigma_{k-1}}$, and evolves independently of $X$. Therefore, we obtain
\begin{align*}
        \prcond{X_{\sigma_k}\in \til{\xi}}{X,\tree}{} = \prcond{X_{\sigma_k}\in \xi(k)}{X,\tree}{}\prcond{X_{\sigma_{k-1}}\in \til{\xi}}{X,\tree}{},
\end{align*}
and hence this gives for all $k\geq 2$
\[
Y_k= -\log\prcond{X_{\sigma_k}\in \xi(k)}{X,\tree}{} = -\log \prcond{X_{\sigma_k}\in \xi(k)}{(X_t)_{t\geq \sigma_{k-1}},\tree(X_{\sigma_{k-1}})}{} . 
\]
Since $Y_k$ is a measurable function of $((X_t)_{t\geq \sigma_{k-1}},\tree(X_{\sigma_{k-1}}))$, using Lemma~\ref{lem:regenerationtimes} we conclude that even conditional on $\{\B_K(\rho)=T_0\}$, the sequence $(Y_k)_{k\geq 2}$ is stationary. 

We now prove the bound on the moments of~$Y_2$. The moments of $|Y_1|$ can be bounded using similar arguments. Let $\til{X}$ be a simple random walk on $\tree(X_{\sigma_1})^a$ started from $X_{\sigma_1}$ and conditioned on never visiting $X_{\sigma_1}^a$.  Let $\til{\sigma}_1$ be the first regeneration time of~$\til{X}$, i.e.\ the first time that $\til{X}$ crosses a long range edge for the first and last time. It is convenient to identify $X_{\sigma_1}^a$ with the parent of $X_{\sigma_1}$ in $T$, so that $\til{X}$ is a walk on a subgraph of $T$ (one can even define  $\til{X}_k=X_{\sigma_1+k}$ for all $k \ge 0$, and then $\til{X}_{\til{\sigma}_1}=X_{\sigma_2}$). By Lemma~\ref{lem:regenerationtimes} we get

%
%
%
\begin{align*}
        &\econd{(Y_2)^\ell}{\B_K(\rho)=T_0} \\&= \E{\sum_{x\in \tree(X_{\sigma_1})}   \prcond{\til{X}_{\til{\sigma}_1}=x}{\tree(X_{\sigma_1})}{}\left( -\log \prcond{x\in \xi(2)}{\tree(X_{\sigma_1})}{}    \right)^\ell} 
        \end{align*}
and similarly
\begin{align*}
        & \econd{\left( -\log \prcond{X_{\sigma_0}\in \til{\xi}}{X,T}{} \right)^\ell}{\B_K(\rho)=T_0} \\&=\econd{ \sum_{x\in \partial T_0}  \prcond{{X}_{\sigma_0}=x}{\B_K(\rho)=T_0}{}\left( -\log \prcond{\til{\xi}_{K-1}=x}{\tree}{}    \right)^\ell}{\B_K(\rho)=T_0},
\end{align*}
where we write $\partial T_0=\{x: d(\rho,x)=K\}$. Write $\prstart{\cdot}{\mathfrak t}$ for the probability measure when $\tree=\mathfrak{t}$. Abusing notation, when considering $\til{X}$ and $\xi(2)$ we also write  $\prstart{\cdot}{\mathfrak t}$ for the probability measure when $\tree(X_{\sigma_1})$ is given by $\mathfrak t$.  It suffices to prove that for all $\ell\in \N$ there exists a positive constant $C$ so that for every realisation $\mathfrak t$ of $\tree(X_{\sigma_1})$ and every realisation $\mathfrak t'$ of $\tree$ with~$\B_K(\rho)=T_0$ we have 
\begin{align}\label{eq:goalfixedtree}
        &\sum_{x\in\mathfrak t}  \prstart{\til{X}_{\til{\sigma}_1}=x}{\mathfrak t}\left( -\log \prstart{x\in \xi(2)}{\mathfrak t}    \right)^\ell\leq C  \\ \label{eq:goalsecondfixed}&\sum_{x\in\mathfrak t'}  \prstart{{X}_{\sigma_0}=x}{\mathfrak t'}\left( -\log \prstart{\til{\xi}_{K-1}=x}{\mathfrak t'}    \right)^\ell\leq C (RK)^\ell,
\end{align}
where $x\in \mathfrak t$ (resp.\ $x\in \mathfrak t'$) ranges over long range edges of $\mathfrak t$ (resp.\ $\mathfrak t'$).
We start by proving~\eqref{eq:goalfixedtree}. 
Let $X'$ be a simple random walk on $\mathfrak t^a$ started from the root $\rho'$ of~$\mathfrak t$    We denote the first regeneration time of $X'$ by $\sigma_1'$. Note that it suffices to prove~\eqref{eq:goalfixedtree} for the walk $X'$, since using the definition of $\til{X}$ we obtain for a positive constant $c$ that 
\begin{align*}
                \prstart{\til{X}_{\til{\sigma}_1}=x}{\mathfrak t} \leq \frac{\prstart{{X'}_{\sigma_1'}=x}{\mathfrak t}}{\prcond{\tau_{(\rho')^a}=\infty}{X_0'=\rho'}{\mathfrak t}} \leq \frac{1}{c}\cdot \prstart{{X}_{\sigma_1'}'=x}{\mathfrak t},
\end{align*}
where the last inequality follows from Lemma~\ref{lem:uniformdrift}. 

We write $d_g(a,b)$ for the graph distance between $a$ and $b$ in the graph $\mathfrak t$, i.e.\ not counting only the long range edges as for $d(a,b)$. 
For every $r$ we set 
\[
A_{r} = \{w=(w_1,w_2)\in \mathfrak t:  d_g(\rho, w_1) =r\},
\]
where again $w$ ranges over long range edges of $\mathfrak t$.
Then we have 
\begin{align*}
       &\sum_{x\in\mathfrak t}  \prstart{{X}_{\sigma'_1}'=x}{\mathfrak t}\left( -\log \prstart{x\in \xi(2)}{\mathfrak t}    \right)^\ell \\&=  \sum_r \sum_{w\in A_{r}}  \prstart{{X}_{\sigma'_1}'=w}{\mathfrak t}\left( -\log \prstart{w\in \xi(2)}{\mathfrak t}    \right)^\ell.
\end{align*}
 The proof of~\eqref{eq:goalfixedtree} will be complete once we show the existence of two positive constants $c_1$ and~$c_2$ so that for all $r$ and all $w\in A_{r}$ 
 \begin{align}\label{eq:newgoaltreefixed}
        \prstart{w\in \xi(2)}{\mathfrak t} \geq c_{1} e^{-c_1 r} \quad \text{ and } \quad \prstart{X_{\sigma'_1}' \in A_{r}}{\mathfrak t}\leq e^{-c_2r}.
 \end{align}
For the first bound, take a path of vertices that connect $\rho$ to $w$. The probability that this is the path taken by the walk that generates the loop erasure is at least $e^{-c_1 r}$ for a positive constant~$c_1$. Indeed, this follows from the bounded degree assumption. Now, by Lemma~\ref{lem:uniformdrift}, once $w$ is reached by the walk, the probability that it is in $\xi(2)$ is at least $c_1$. 
For the second bound in~\eqref{eq:newgoaltreefixed}, using that $\sigma'_1$ has exponential tails from Remark~\ref{rem:everyreal} we have 
\[
\prstart{X'_{\sigma'_1} \in A_{r}}{\mathfrak t}\leq \prstart{\sigma'_1\geq r}{\mathfrak t} \leq e^{-c_2 r}
\]
for a positive constant $c_2$.

For the proof of~\eqref{eq:goalsecondfixed}, note that $\prstart{\til{\xi}_{K-1}=x}{\mathfrak t'} \geq e^{-c_3 RK}$ for a positive constant $c_3$, since we can take a path of long range edges of length $K$ and require that the walk creating the loop erasure takes this path and then escapes, similarly to the proof of the first inequality in~\eqref{eq:newgoaltreefixed}. So we  now get that 
\[
\sum_{x\in\mathfrak t'}  \prstart{{X}_{\sigma_0}=x}{\mathfrak t}\left( -\log \prstart{\til{\xi}_{K-1}=x}{\mathfrak t'}    \right)^\ell \lesssim (RK)^\ell \sum_{x\in \mathfrak t'}\prstart{X_{\sigma_0}=x}{\mathfrak t'} = (RK)^\ell.
\]
It remains to prove~\eqref{eq:varofsumofy}. To simplify notation, we write $\prstart{\cdot}{T_0}$ for the probability measure conditional on $\B_K(\rho)=T_0$ and similarly $\estart{\cdot}{T_0}$, $\rm{Var}_{T_0}$ and $\cov_{T_0}$. With this notation we have 
 \[
 {\rm{Var}}_{T_0}\left(\sum_{i=1}^{k}  Y_i \right) = \sum_{i=1}^{k} \Var_{T_0}(Y_i) + 2\sum_{i<j} \cov_{T_0}( Y_i, Y_j).
 \]
 Using~\eqref{eq:momentsofy} we get that $
 \sum_{i=1}^{k} \Var_{T_0}(Y_i) \lesssim k$,
 and hence it suffices to prove that 
 \begin{align}\label{eq:decorrelateyi}
        \sum_{i<j} \cov_{T_0}( Y_i, Y_j) \lesssim k.
 \end{align}
In order to prove this, for $j>i$ we are going to define random variables $Y_{i,j}$ and events $B(i,j)$ so that 
\begin{enumerate}
        \item[(i)] $Y_{i,j}\1(B(i,j))$ and $B(i,j)$ are independent of $Y_j$,
        \item[(ii)] $\prstart{B(i,j)^c}{T_0}\leq e^{-c(j-i)}$ for a positive constant $c$ and
        \item[(iii)] $|Y_i-Y_{i,j}|\1(B(i,j)) \leq e^{-c' (j-i)}$ for another positive constant $c'$.
\end{enumerate}
Therefore, assuming that we have defined $Y_{i,j}$ and $B(i,j)$ satisfying the above conditions we can finish the proof, since
\begin{align*}
        \cov_{T_0}(Y_i,Y_j) = &\estart{(Y_i-\estart{Y_i}{T_0})(Y_j-\estart{Y_j}{T_0})\1(B(i,j))}{T_0} \\+& \estart{(Y_i-\estart{Y_i}{T_0})(Y_j-\estart{Y_j}{T_0})\1(B^c(i,j))}{T_0}\\ \lesssim &
        \estart{(Y_i-\estart{Y_i}{T_0})(Y_j-\estart{Y_j}{T_0})\1(B(i,j))}{T_0} + e^{-c_1(j-i)/(2C)}\\
         = &\estart{(Y_i- Y_{i,j})Y_j\1(B(i,j))}{T_0} - \estart{(Y_i-Y_{i,j})\1(B(i,j))}{T_0}\estart{Y_j}{T_0} \\&+ e^{-c_1(j-i)/(2C)},
\end{align*}
where for the inequality we used Cauchy Schwarz together with~\eqref{eq:momentsofy} and (ii) and for the last equality we used (i). Using~\eqref{eq:momentsofy} and (iii) gives
\begin{align*}
\estart{(Y_i- Y_{i,j})Y_j\1(B(i,j))}{T_0} \lesssim e^{-c''(j-i)} \ \text{ and }  \\  \estart{(Y_i-Y_{i,j})\1(B(i,j))}{T_0}\estart{Y_j}{T_0}\lesssim e^{-c''(j-i)}.
\end{align*}
Taking the sum over $j>i$ yields~\eqref{eq:decorrelateyi} and finishes the proof. So we now turn to define $Y_{i,j}$ and $B(i,j)$ for $j>i$.
 
For each $i$ let $X^i$ be the walk that generates the loop erased path $\xi(i)$, i.e.\ $X^i$ is a simple random walk in the subtree $\tree(X_{\sigma_{i-1}})$ started from $X_{\sigma_{i-1}}$ and $\xi(i)$ is obtained by only considering the times when $X^i$ crosses long range edges and erasing loops in the chronological order in which they are created. Now for $i<j$ we let $\xi(i,j)$ be the loop erased path (across long range edges) obtained from the path $X^i$ when we run it until the first time that $X^i$ reaches the level of $X_{\sigma_{j-1}}$.  We set 
\begin{align*}
Z_i =\prcond{X_{\sigma_i}\in \xi(i)}{\tree(X_{\sigma_{i-1}}), X}{T_0},  \ & Z_{i,j} =\prcond{X_{\sigma_i}\in \xi(i,j)}{\tree(X_{\sigma_{i-1}}), X}{T_0} \\ &\text{ and } \ Y_{i,j}= -\log Z_{i,j}.
\end{align*}
Note that by the definition of $\xi(i,j)$ we have that 
\[
Z_{i,j} =\prcond{X_{\sigma_i}\in \xi(i,j)}{\tree(X_{\sigma_{i-1}})\setminus \tree(X_{\sigma_{j-1}}), (X_t)_{t=\sigma_{i-1}}^{\sigma_{j-1}}}{T_0}
\]
Let $A(i,j)$ be the event that $X^i$ returns to  $X_{\sigma_i}$ after reaching the level of $X_{\sigma_{j-1}}$ for the first time. Then we have 
\begin{align*}
|Z_i - Z_{i,j}| = |&\prcond{X_{\sigma_i}\in \xi(i)}{\tree(X_{\sigma_{i-1}}),X}{T_0}  - \prcond{X_{\sigma_i}\in \xi(i,j)}{\tree(X_{\sigma_{i-1}}),X}{T_0}| \\
= |&\prcond{X_{\sigma_i}\in \xi(i), X_{\sigma_i}\notin \xi(i,j)}{\tree(X_{\sigma_{i-1}}),X}{T_0}  
\\&-\prcond{X_{\sigma_i}\notin \xi(i), X_{\sigma_i}\in \xi(i,j)}{\tree(X_{\sigma_{i-1}}),X}{T_0}|\\
&\leq \prcond{A(i,j)}{\tree(X_{\sigma_{i-1}}), X}{T_0}.
\end{align*}
Using Lemma~\ref{lem:uniformdrift} we obtain that there exists a positive constant $c$ so that 
\[
\prcond{A(i,j)}{\tree(X_{\sigma_{i-1}}),X}{T_0} \leq e^{-c(j-i-1)}.
\]
Using that $|\log x - \log y|\leq |x-y|/(x\wedge y)$ we now obtain \begin{align*}
        |Y_{i,j}-Y_i|=|\log Z_{i,j}-\log Z_i | \leq \frac{|Z_{i,j} - Z_i|}{Z_{i,j}\wedge Z_i}
\end{align*}
Let $B(i,j)=\{ d_g(X_{\sigma_i},X_{\sigma_{i-1}}) \leq \lfloor(j-i)/C\rfloor\} $ for a large positive constant $C$. On $B(i,j)$ we have 
\[
Z_{i,j}\wedge Z_i \geq c (\Delta+1)^{-(j-i)/C},
\]
where $c$ is the positive constant from Lemma~\ref{lem:uniformdrift} and $\Delta$ is the maximum degree. Indeed, the right hand side above is a lower bound on the probability that $X^i$ visits $X_{\sigma_i}$ without backtracking until the first such visit and then escapes. Therefore, choosing $C$ sufficiently large we get that 
\begin{align}\label{eq:boundondiftimesb}
| Y_{i,j}-Y_i |\1(B(i,j)) \leq e^{-c'' (j-i)},
\end{align}
where $c''$ is a positive constant. Using next Lemma~\ref{lem:regenerationtimes} we get that for a positive constant $c_1$ 
\begin{align}\label{eq:boundonbij}
\pr{B(i,j)^c}\leq e^{-c_1 \lfloor(j-i)/C\rfloor}.
\end{align}
Finally we note that $Y_{i,j}\1(B(i,j))$ and $B(i,j)$ are independent of~$Y_j$, since they depend on independent parts of the tree by the definition of regeneration times. This finishes the proof.
\end{proof}

\begin{proposition}\label{pro:entropy}
        Let $\tree$ be a {\quasi} as in Definition~\ref{def:tree} and let $\xi$ and $\til{\xi}$ be two independent loop erased random walks on $\tree$ both started from the root. Then there exists a positive constant~$\mathfrak{h}=\mathfrak{h}_n$ so that almost surely
        \begin{align*}
                \frac{-\log \prcond{\xi_k\in \til{\xi}}{T,\xi}{} }{k}\to \mathfrak{h} \text{ as } k\to\infty.
        \end{align*}
        Fix $K\geq 0$ and let $T_0$ be a realisation of the first $K$ levels of $\tree$. For all $\epsilon>0$, there exists a positive constant $C$ so that for all $k\geq(R K)^2$
        \begin{align*}
                \prcond{\left|-\log \prcond{\xi_k\in \til{\xi}}{T,\xi}{} -\mathfrak{h}k\right|>C\sqrt{k}}{\B_K(\rho)=T_0}{}\leq \epsilon.
        \end{align*}
\end{proposition}

\begin{proof}[\bf Proof]
Again to simplify notation we write $\prstart{\cdot}{T_0}$ for the probability measure $\prcond{\cdot}{\B_K(\rho)=T_0}{}$.
Let $X$ be the simple random walk on $\tree$ that generates the loop erasure $\xi$ and let $(\sigma_k)_{k \ge 1}$ be its regeneration times after time $\sigma_0$ and $\sigma_0$ be the hitting time of $\partial T_0$ as in Lemma~\ref{lem:regenerationtimes}. Then we get 
\begin{align*}
        -\log \prcond{(X_{\sigma_k-1}, X_{\sigma_k})\in \til{\xi}}{X,\tree}{} = \sum_{i=1}^{k} Y_i - \log \prcond{(X_{\sigma_0-1},X_{\sigma_0})\in \til{\xi}}{X,\tree}{},
        \end{align*} 
        where $Y_i$ are the variables of Lemma~\ref{lem:iidentropy} and which are stationary for $i\geq 2$ and $\E{|Y_i|}\leq  C$ for all~$i\geq 1$. Therefore, applying the ergodic theorem and using also~\eqref{eq:tau0bound} we deduce that there exists a constant~$\gamma$ so that almost surely
\[
-\frac{\log\prcond{(X_{\sigma_{k}-1},X_{\sigma_k})\in \til{\xi}}{X,\tree}{} }{k} \to \gamma \quad \text{ as } k\to \infty.
\]
Let $\phi_k = d(\rho,X_{\sigma_k})$. Then $\xi_{\phi_k} = (X_{\sigma_k-1},X_{\sigma_k})$, and hence from the above almost surely as $k\to\infty$ 
\[
-\frac{\log \prcond{\xi_{\phi_k}\in \til{\xi}}{X,\tree}{}}{k} \to \gamma.
\]
Lemma~\ref{lem:regenerationtimes} now gives that almost surely $\phi_k/k\to \E{\phi_2-\phi_1}$ as $k\to\infty$ with $\E{\phi_2-\phi_1}<\infty$. This now implies that 
\[
-\frac{\log \prcond{\xi_{k}\in \til{\xi}}{\xi,\tree}{}}{k} \to \frac{\gamma}{\E{\phi_2-\phi_1}}=:\mathfrak{h}.
\]
We turn to the proof of the fluctuations. Using the bound on the variance of $\sum_{i=1}^{k}Y_i$ from Lemma~\ref{lem:iidentropy} together with~\eqref{eq:tau0bound} and Chebyshev's inequality we obtain that for all $\epsilon>0$ there exists a positive constant~$C$ so that for all $k\geq (KR)^2$
\begin{align*}
        \prstart{\left| -\log \prcond{\xi_{\phi_k} \in \til{\xi}}{\xi,\tree}{} - \gamma k \right|\geq C\sqrt{k}}{T_0}\leq \epsilon.
\end{align*}
We now need to transfer the fluctuations result to the process $-\log \prcond{\xi_k\in \til{\xi}}{\xi,\tree}{}$. 
As in Claim~\ref{cl:renewal} for each $k\in \N$ let 
\[
{N}_k = \max\{ i\geq 0: \phi_i\leq k+K\}.
\]
Then we have 
\begin{align*}
       & \prstart{\left| -\log \prcond{\xi_k\in \til{\xi}}{\xi,\tree}{} - \mathfrak{h}k\right|>C\sqrt{k}}{T_0}   \\ \leq &\prstart{-\log \prcond{\xi_{\phi_{{N}_k+1}}\in \til{\xi}}{\xi,\tree}{}>\mathfrak{h}k + C\sqrt{k}}{T_0} \\&+ \prstart{-\log \prcond{\xi_{\phi_{{N}_k}}\in \til{\xi}}{\xi,\tree}{}<\mathfrak{h}k - C\sqrt{k}}{T_0}.
\end{align*}
Using again the monotonicity, in the sense that if $\xi_i\in \til{\xi}$, then also $\xi_{j}\in \til{\xi}
$ for every $j<i$, and the concentration of ${N}_k$ from Claim~\ref{cl:renewal} proves the result for the suitable choice of the constant $C$. 
\end{proof}

\begin{remark} \label{rem:BSlimit}
        \rm{
        We note that both the entropy constant $\mathfrak{h}$ and the speed constant $\nu$ appearing in Proposition~\ref{pro:entropy} and Lemma \ref{lem:speed} depend on $n$ but are both of order $1$. We recall that by the bounded degree assumption  $(G_n)$ has a subsequence converging in   the Benjamini-Schramm sense. To prove cutoff w.h.p.\ it suffices to show that any subsequence has a further subsequence for which cutoff holds w.h.p. We may thus assume such a limit exists. One can show that if $(G_n)$ has a Benjamini-Schramm limit then the entropy and speed constants converge to the corresponding constants when the quasi tree $\tree$ is defined w.r.t.\ the limit (with $R=\infty$). In general, the rate of convergence  can be arbitrary, and so in order to obtain any control on the cutoff window it is important to work with our $\nu$ and $\mathfrak{h}$, rather than with their limit.}
\end{remark}

\section{Truncation}\label{sec:truncation}

\begin{definition}\label{def:wtil}
        \rm{
        Let $e$ be a long range edge of $\tree$ and let $\xi$ be a loop erased random walk started from the root of $\tree$ as in Definition~\ref{def:lerw}. We define 
        \[
W_\tree(e) = -\log \prcond{e\in \xi}{\tree}{}.
\]
For a long range edge $e=(x,y)$ with $d(\rho,x)<d(\rho,y)$ we write $\ell(e)=d(\rho,y)$.  We define
\[
\til{W}_\tree(e) = -\log \prcond{(X_{\tau^{\ell(e)}-1},X_{\tau^{\ell(e)}})=e}{\tree}{},
\]
where $X$ is a simple random walk on $\tree$ started from the root and $\tau^{\ell(e)}=\inf\{t\geq 0: d(\rho,X_t)=\ell(e)\}$.         }
\end{definition}

\begin{remark}
        \rm{
        In the definition $\til{W}$ above, we are requiring the walk $X$ to first hit level $\ell(e)$ by crossing $e$. Note that in this way $\til{W}$ only depends on the first~$\ell(e)$ levels of the tree. 
        }
\end{remark}

\begin{lemma}\label{lem:twolooperasures}
        There exists a positive constant $c$ so that for all realisations of $\tree$ and all edges $e$ of~$\tree$ we have 
        \[
      W_\tree(e) \geq  \til{W}_\tree(e)- cR^2.
        \]
\end{lemma}

\begin{proof}[\bf Proof]

In this proof we fix the graph $\tree$ and so we drop the dependence on $\tree$ from the notation. 

Let $X$ be a simple random walk on $\tree$ started from the root $\rho$ and let $\xi(e)$ be the loop erasure of the path of $X$ until the first time that it hits level $\ell(e)$. Then we clearly have 
\[
\{(X_{\tau^{\ell(e)}-1},X_{\tau^{\ell(e)}})=e\} = \{e\in \xi(e)\}.
\]
It suffices to show that there exists a positive constant $c$ so that 
\begin{align}\label{eq:goalloop}
        \prstart{e\in \xi(e)}{} \gtrsim e^{-cR^2}\cdot \prstart{e\in \xi}{},
\end{align}
since taking logarithms of both sides proves the lemma.
To simplify notation we write $\ell=\ell(e)$ (and $\tau^{\ell}=\inf\{t\geq 0: d(\rho,X_t)=\ell \}$ as above). Let $e_1,\ldots, e_\ell=e$ be the sequence of long range edges leading to $e$.
Letting $e_i= (x_{i},y_{i})$ with $d(\rho,x_i)<d(\rho,y_i)$ and using Lemma~\ref{lem:domainmarkov} for the transition probabilities of the loop erased random walk we now get
\begin{align}\label{eq:looperasurestopped}
        \prstart{e\in \xi(e)}{} = \prod_{i=0}^{\ell-1} \prstart{\til{X}^i_{L_i^\ell}=x_{i+1}}{y_{i}},
\end{align}
where $y_0=\rho$ and for each $i$, $\til{X}^i$ is a simple random walk on $\tree(y_i)$ and 
$L_i^\ell$ denotes the last time before reaching level $\ell$ of $\tree$ that $\til{X}^i$ is in the ball centred at $y_i$. Similarly for the loop erasure~$\xi$ we have 
\begin{align}\label{eq:loooperasurecontinue}
        \prstart{e\in \xi}{} = \prod_{i=0}^{\ell-1}\prstart{\til{X}^i_{L_i}=x_{i+1}}{y_{i}},
        \end{align}
        where now $L_i$ is the last time that $\til{X}$ is in the ball centred at $y_i$. 

Using the last exit decomposition formula, we obtain
\begin{align}\label{eq:twoequations}
\begin{split}
\prstart{\til{X}^i_{L_i^\ell}=x_{i+1}}{y_{i}} &= \frac{\prstart{\tau_{x_{i+1}}<\tau^\ell}{y_i}}{\prstart{\tau^\ell<\tau_{x_{i+1}}^+}{x_{i+1}}}\cdot P(x_{i+1}, y_{i+1})\cdot   \prstart{\tau_{x_{i+1}}>\tau^\ell}{y_{i+1}} 
\\
\prstart{\til{X}^i_{L_i}=x_{i+1}}{y_{i}} &= \frac{\prstart{\tau_{x_{i+1}}<\infty}{y_i}}{\prstart{\tau_{x_{i+1}}^+=\infty}{x_{i+1}}}\cdot P(x_{i+1}, y_{i+1}) \cdot \prstart{\tau_{x_{i+1}}=\infty}{y_{i+1}}  
\end{split}
\end{align}
where $\frac{1}{\prstart{\tau_{x_{i+1}}^+=\infty}{x_{i+1}}}$ is the expected number of visits to $x_{i+1}$, once it is reached, while  $\frac{1}{\prstart{\tau^\ell<\tau_{x_{i+1}}^+}{x_{i+1}}}$  is the expected number of such visits before time $\tau^{\ell}$.  We need to compare the ratios of the terms appearing in the two expressions above. 

For the last two terms we have $\mathbb{P}_{y_{i+1}}(\tau_{x_{i+1}}>\tau^{\ell}) \ge 
\mathbb{P}_{y_{i+1}}(\tau_{x_{i+1}} = \infty)$. We now explain that it suffices to prove that there exist constants $c_1, c_2$ and $c_3$ so that for all $i\leq \ell$
\begin{align}\label{eq:tauelltauxi1}
\frac{\mathbb{P}_{x_{i+1}} (\tau_{x_{i+1}}^+=\infty)}{
\mathbb{P}_{x_{i+1}}(\tau_{x_{i+1}}^+>\tau^{\ell})} \geq \frac{1}{1+c_2e^{-c_1(\ell-i)}},
\end{align}
for $i\leq \ell-c_3R$
\begin{align}\label{eq:boundforsmalli}
        \frac{\mathbb{P}_{y_i}(\tau_{x_{i+1}}<\tau^{\ell})}{
\mathbb{P}_{y_i} (\tau_{x_{i+1}}<\infty)} \geq \frac{1}{1-e^{-c_1(\ell-i)/2}},
\end{align}
while for all $i\leq \ell$
\begin{align}\label{eq:alli}
\frac{\mathbb{P}_{y_i}(\tau_{x_{i+1}}<\tau^{\ell})}{
\mathbb{P}_{y_i} (\tau_{x_{i+1}}<\infty)} \geq \frac{1}{(\Delta+1)^R}.  
\end{align}
Indeed, once these bounds are established, we can easily finish the proof, since for all $i$ satisfying $\ell-c_3R< i\leq \ell$ plugging the bounds~\eqref{eq:tauelltauxi1} and~\eqref{eq:alli} into~\eqref{eq:twoequations} we get
        \begin{align*}
                \prstart{\til{X}^i_{L_i^\ell}=x_{i+1}}{y_{i}} \geq \frac{1}{1+c_2e^{-c_1(\ell-i)}}\cdot \frac{1}{(\Delta+1)^R}\cdot   \prstart{\til{X}^i_{L_i}=x_{i+1}}{y_{i}}.
        \end{align*}
        For $i$ satisfying $i\leq \ell - c_3R$ plugging the bounds~\eqref{eq:tauelltauxi1} and~\eqref{eq:boundforsmalli} into~\eqref{eq:twoequations} gives
        \begin{align*}
                \prstart{\til{X}^i_{L_i^\ell}=x_{i+1}}{y_{i}} \geq \frac{1-e^{-c_1(\ell-i)/2}}{1+c_2e^{-c_1(\ell-i)}}\cdot   \prstart{\til{X}^i_{L_i}=x_{i+1}}{y_{i}}.
        \end{align*}
        From these two inequalities together with~\eqref{eq:looperasurestopped} and~\eqref{eq:loooperasurecontinue} we now deduce
        \begin{align*}
         \prstart{e\in \xi(e)}{} \geq \frac{1}{(\Delta+1)^{c_3R^2}}\cdot  \prstart{e\in \xi}{} \cdot \prod_{i=0}^{\ell-1}\frac{1-e^{-c_1i/2}}{1+c_2e^{-c_1 i}}\gtrsim \frac{\prstart{e\in \xi}{}}{(\Delta+1)^{c_3R^2}},
        \end{align*}
        which proves~\eqref{eq:goalloop}, and hence finishes the proof of the lemma. It thus remains to prove~\eqref{eq:tauelltauxi1}, \eqref{eq:boundforsmalli} and~\eqref{eq:alli}.

We start with~\eqref{eq:tauelltauxi1}. 
We have 
\begin{align*}
\prstart{\tau^\ell<\tau^+_{x_{i+1}}}{x_{i+1}}=\prstart{\tau_{x_{i+1}}^+=\infty}{x_{i+1}}  + \prstart{\tau^\ell<\tau_{x_{i+1}}^+<\infty}{x_{i+1}}.
\end{align*}    
Using Lemma~\ref{lem:uniformdrift} we get that there exists a positive constant $c_1$ such that 
        \[
\prstart{\tau^\ell<\tau_{x_{i+1}}^+<\infty}{x_{i+1}} \leq  e^{-c_1(\ell-i)}.
        \]
        Using Lemma~\ref{lem:uniformdrift} again, we get that there exists a positive constant $c_2$ so that
        \begin{align*}               \prstart{\tau^\ell<\tau_{x_{i+1}}^+}{x_{i+1}}\leq  \prstart{\tau_{x_{i+1}}^+=\infty}{x_{i+1}}  \left(1+ c_2e^{-c_1(\ell-i)} \right),
        \end{align*}
        therefore establishing~\eqref{eq:tauelltauxi1}.
 
       Suppose that $i$ is such that $i\leq \ell- c_3R$ for a positive constant $c_3$ to be determined later. Then using Lemma~\ref{lem:uniformdrift} we have 
       \begin{align*}
       \prstart{\tau_{x_{i+1}}<\infty}{y_{i}} &=\prstart{\tau_{x_{i+1}}<\tau^\ell}{y_{i}} +  \prstart{\tau^\ell<\tau_{x_{i+1}}<\infty}{y_{i}} \\&\leq  \prstart{\tau_{x_{i+1}}<\tau^\ell}{y_{i}} + e^{-c_1(\ell -i)} \\ &\leq  \prstart{\tau_{x_{i+1}}<\tau^\ell}{y_{i}}  + e^{-c_1(\ell-i)/2} \cdot \frac{1}{(\Delta+1)^R},
       \end{align*}
        where the last inequality follows by choosing $c_3$ as a function of $\Delta$ and $c_1$. Using that 
        \[
         \prstart{\tau_{x_{i+1}}<\infty}{y_{i}} \geq \frac{1}{(\Delta+1)^R},
       \] 
        we obtain
        \begin{align*}
          \prstart{\tau_{x_{i+1}}<\infty}{y_{i}} \leq  \prstart{\tau_{x_{i+1}}<\tau^\ell}{y_{i}} + e^{-c_1(\ell-i)/2}  \prstart{\tau_{x_{i+1}}<\infty}{y_{i}}.
        \end{align*}
        Rearranging this gives
                \begin{align*}
       (1-e^{-c_1(\ell-i)/2} ) \prstart{\tau_{x_{i+1}}<\infty}{y_{i}} \leq  \prstart{\tau_{x_{i+1}}<\tau^\ell}{y_{i}},
        \end{align*}
        thus establishing~\eqref{eq:boundforsmalli}.
        
        Finally for all $i$ we also have 
        \begin{align*}
                \prstart{\tau_{x_{i+1}}<\tau^\ell}{y_{i}} \geq \frac{1}{(\Delta+1)^R} \geq \frac{1}{(\Delta+1)^R}\cdot \prstart{\tau_{x_{i+1}}<\infty}{y_{i}},
        \end{align*}
        proving~\eqref{eq:alli}. This completes the proof of the lemma.
\end{proof}

\begin{definition}\label{def:trunc}\rm{
Let $A>0$ and $K=\lceil C_2\log \log n \rceil$ for a constant $C_2$ to be determined. For a long range edge $e$ of $\tree$ we define the ``truncation event'' $\trunc{e}{A}$ to be 
\[
\trunc{e}{A} = \left\{ \til{W}_\tree(e) > \log n  -  A\sqrt{\log n}\right\}\cap \{\ell(e)\geq K\},
\]
where $\ell(e)$ stands for the level of $e$. 
}       
\end{definition}

In the next section, where we construct the coupling of the walk on $\tree$ with the walk on $G_n^*$ we will need to truncate the edges of $\tree$ that satisfy the ``truncation criterion'' above. We will then need to ensure that the random walk on $\tree$ does not visit  truncated edges by the relevant time $t$ with large probability. We achieve this in the following lemma.

%

%

\begin{lemma}\label{lem:truncation}
Let $K$ be as in Definition~\ref{def:trunc} and let $T_0$ be a realisation of the first $K$ levels of $\tree$. Let~$X$ be a simple random walk on $\tree$ started from its root and set $t= \frac{\log n}{\nu\mathfrak{h}} -B\sqrt{\log n}$, where $\nu $ and $\mathfrak{h}$ are given in Lemma~\ref{lem:speed} and Proposition~\ref{pro:entropy} respectively. Then for all $\epsilon\in (0,1)$ there exist $B$ and $A$ (depending on $\epsilon$ and $B$) sufficiently large so that
\[
\prcond{\bigcup_{k\leq t}\trunc{(X_{k-1},X_k)}{A}}{\B_K(\rho)=T_0}{} <\epsilon.
\]      
\end{lemma}

\begin{proof}[\bf Proof]
Using Lemma~\ref{lem:speed} there exists a positive constant $C$ so that if 
\[
D=\left\{\sup_{s\leq t}d(\rho,X_s)\leq \nu t +C\sqrt{t}\right\},
\]
then $\pr{D}\geq 1-\epsilon$. To simplify notation, we write again $\prstart{\cdot}{\tree_0}$ for the probability measure $\prcond{\cdot}{\B_K(\rho)=\tree_0}{}$.
We now get
\begin{align*}
        \prstart{\bigcup_{k\leq t}\trunc{(X_{k-1},X_k)}{A}}{T_0}\leq \prstart{\bigcup_{k\leq t}\trunc{(X_{k-1},X_k)}{A}, D}{T_0} +\epsilon.
\end{align*}
Define $F(e)$ to be the event that~$e$ is the first edge crossed by the walk for which the event $\trunc{e}{A}$ holds. Then we have 
\begin{align*}
        \prstart{\bigcup_{k\leq t}\trunc{(X_{k-1},X_k)}{A}, D}{T_0} \leq \prstart{  \bigcup_{e\in \tree:\ d(\rho, e)\leq \nu t+C\sqrt{t}}F(e)}{T_0}\\=\sum_{e\in \tree:\ d(\rho, e)\leq \nu t+C\sqrt{t}}\prstart{F(e)}{T_0}.
\end{align*}
Let $\xi$ be the loop erasure of $X$ (considered when it crosses long range edges) and define $\til{F}(e)$ to be the event that $e$ is the first long range edge crossed by the loop erasure $\xi$ for which the event $\trunc{e}{A}$ holds.  For a long range edge $e=(e_-,e_+)$, let $\til{\tau}(e)$ be the first return time to $e_+$ by $X$ after the first time $X$ crosses $e$. Then for every realisation $\mathfrak t$ of $\tree$ for which $\B_K(\rho)=T_0$ we have 
        \[
        \prstart{F(e), \til{\tau}(e)=\infty}{\mathfrak t} = \prstart{F(e), e\in \xi}{\mathfrak t} \leq \prstart{\til{F}(e), e\in \xi}{\mathfrak t},
        \]
        where in the notation above we have fixed $\tree$ to be $\mathfrak t$.
        By Lemma~\ref{lem:uniformdrift} we now get
        \[
        \prstart{F(e)}{\mathfrak t} \lesssim \prstart{F(e), \til{\tau}(e)=\infty}{\mathfrak t},
        \]
        and hence putting all things together we deduce
        \begin{align*}
                &\prstart{\bigcup_{k\leq t}\trunc{(X_{k-1}, X_k)}{A}, D}{T_0}\lesssim 
        \sum_{e\in \tree:\ d(\rho, e)\leq \nu t+C\sqrt{t}} \prstart{\til{F}(e)}{T_0}  \\&=
                \prstart{\bigcup_{e\in \tree: \ d(\rho, e)\leq \nu t+C\sqrt{t}} \til{F}(e)}{T_0} = \prstart{\bigcup_{k\leq \nu t+C\sqrt{t}}\trunc{\xi_k}{A}}{T_0},
        \end{align*}
        where the first equality follows since by definition the events $\til{F}(e)$ are disjoint.
        By Lemma~\ref{lem:twolooperasures} we have that on the event $\trunc{\xi_k}{A}$
        \[
        W_\tree(\xi_k) >\log n- A\sqrt{\log n} -cR^2.
        \]
        Using that $W_\tree(\xi_k)\leq W_\tree(\xi_{k+1})$ (since the loop erasure is only considered when it crosses long range edges) gives that on the event $\cup_{k\leq L}\trunc{\xi_k}{A}$ with $L=\nu t +C\sqrt{t}$ we have 
        \[
        W_\tree(\xi_L) > \log n-A\sqrt{\log n}-cR^2.
        \]
        This together with Proposition~\ref{pro:entropy} conclude the proof.
\end{proof}

\section{Coupling}
\label{s:coupling}
Recall that we refer to the edges of the perfect matching of $\G$ as \emph{long range edges}.
\begin{definition}
\rm{
In the graph $\G$ we define the (long range) distance between $x$ and $y$ to be the minimal number of long range edges needed to cross to go from $x$ to $y$, when we only allow at most~$R$ consecutive  edges of $G_n$ in the path from $x$ to $y$ and we do not allow any long range edge (here considered as undirected) to be crossed more than once. (The first constraint is put in order to avoid having long range distance 0 between all pairs of vertices whose graph distance in $G_n$ is  $R+1$, whereas without the second constraint  the distance between such pairs would be always at most 2.)  Like for the {\quasi} $\tree$, we rarely use the regular graph distance on $G_n^*$, so the term ``distance'' below will refer to the aforementioned distance, unless otherwise specified.

We write $\B_K^*(x)$ to denote the ball of radius~$K$ and centre~$x$ in this metric.   
We write $\B_{G_n}(x,r)$ for a ball centred at $x$ of radius $r$ in the graph metric of $G_n$. When $r=R$, we call it the ${\G}\rball$ centred at $x$.}
\end{definition}

As in Lemma \ref{lem:truncation} we set
\begin{align}\label{eq:defoft}
t=\frac{\log n}{\nu\mathfrak{h}} - B\sqrt{\log n}
\end{align}
for a constant $B$ to be determined and let $A$ be as in Lemma~\ref{lem:truncation}. 

\begin{definition}
\rm{
We call a vertex $x$ a $K$-root of $\G$ if $\B_K^*(x)$ is a possible realisation of the first $K$ levels of the {\quasi}  $\tree$ (corresponding to $G_n$). If $x$ is a $K$-root and $i \le K$, we denote by $\partial \B^*_{i}(x)$ the collection of vertices of (long range) distance~$i$ from $x$. (Note that this is a slight abuse of notation, since $\partial \B^*_i(x)$ is not the internal vertex boundary of $\B^*_i(x)$ as the internal vertex boundary does not contain the centres of the $\tree\rball$s at distance~$i$ from~$x$.)     
} 
\end{definition}

We next define an exploration process of $\G$ and a coupling between the walk $X$ on $\G$ and a walk~$\til{X}$ on the {\quasi} $\tree$ corresponding to $G_n$.

\begin{definition}\label{def:coupling}
        \rm{
        Let $K=\lceil C_2 \log \log n \rceil $ for a constant $C_2$ to be determined as in Definition~\ref{def:trunc}, and suppose we work conditional on the event that $x_0$ is a $K$-root and that $\B_K^*(x_0)=T_0$, where~$T_0$ is a realisation of the first $K$ levels of a {\quasi}.
        Let $\{z_1,\ldots, z_L\}\subseteq \partial\B^*_{K/2}(x_0)$ be the collection of centres of $\tree\rball$s at long range distance $K/2$ from $x_0$, where $L\leq |\partial \B^*_{K/2}(x_0)|$. For each $z\in \partial \B^*_{K/2}(x_0)$ we denote by $V_z$ the set of offspring of $z$ on $\partial\B_K^*(x_0)$.
         Let $z\in \partial \B^*_{K/2}(x_0)$.  We now describe the \textbf{exploration process} of $\G$ corresponding to the set $V_z$ by constructing a coupling of a subset of $\G$ with a subset of a \quasi ~$\tree$ conditioned on the first $K$ levels of $\tree$ being equal to $T_0$. We first reveal all long range edges of $\tree$ with one endpoint in $\partial T_0$, i.e.\ with one endpoint at long range distance $K$ from $x_0$. For the long range edges originating in $V_z$ we couple them with the long range edges of $\G$ by using the optimal coupling between the two uniform distributions at every step. (At every step in $\G$ we choose an endpoint at random among all those that have not been selected yet.) If at some point one of these couplings fails, then we truncate the edge where this happened and stop the exploration for this edge in $\G$ but we continue it in $\tree$. We also truncate an edge and stop the exploration in $\G$ if the ${\G}\rball$ around the newly revealed endpoint of the edge intersects an already revealed ${\G}\rball$ (whenever we reveal the other endpoint of a long range edge, we reveal the ball of radius $R$ around it in the graph metric of $G_n$; coupling this endpoint between~$T$ and~$G_n^*$ is the same as coupling the two $R$-balls). In the case where the $\G\rball$ centred at the endpoint intersects an already revealed $\G\rball$, then we also truncate the edge leading to its centre and stop the exploration there too even though we may have already revealed some of its offspring. We always continue the exploration for $\tree$.  
        Once all long range edges joining levels $K$ and~$K+1$ of~$\tree$ have been revealed, we examine which of those satisfy the truncation criterion $\trunc{e}{A}$ (which is defined w.r.t.\ $\tree$, not $G_n^*$). We then stop the exploration at these edges for the graph $\G$, but we do continue the exploration of their offspring for the \quasi~$\tree$. Suppose we have explored all $k$ level edges of the {\quasi}~$\tree$ and also the corresponding ones in $\G$ that have not been truncated. Then for the edges of level $k+1$ we explore all of them in~$\tree$ and we use the optimal coupling to match the ones that come from non-truncated edges in $\G$ with the corresponding ones of $\tree$. We truncate an edge and stop the exploration process at this edge if the optimal coupling between the two uniform distributions fails at the endpoint of the edge or if the~${\G}\rball$ centred at the endpoint intersects an already revealed ${\G}\rball$. In the case where the $\G\rball$ centred at the endpoint intersects an already revealed $\G\rball$, then we also truncate the edge leading to its centre and stop the exploration there too even though we may have already revealed some of its offspring. We always continue the exploration for $\tree$.
        We continue the exploration process for~$t$ levels.
        
 We now describe a \textbf{coupling} of the walk $X$ on $\G$ starting from $x\in V_z$ with a walk $\til{X}$ on $\tree$ starting from $x$ as follows: we move $X$ and $\til{X}$ together for  $t$ steps as long as none of the following happen:
   \begin{enumerate}[(i)]
                   \item $\til{X}$ crosses a truncated  edge;
                \item There exists a vertex $v$ such that $\til{X}$ visits $v$ and then reaches the internal vertex boundary of the~${\tree}\rball$ centred at $v$ (i.e.\ it reaches a vertex in the ${\tree}\rball$ centred at $v$ which is at distance~$R$ (in the graph metric of $G_n$) from $v$) and does so  by time $t$, or
                \item $\til{X}$ visits a vertex $w\in \partial \B_{K/2}^*(x_0)$ for some $w\neq z$. 
    \end{enumerate}
If none of these occurs by time $t$ we say the \emph{coupling is successful}. 

We write $\F_i$ for the $\sigma$-algebra generated by $T_0$ and the exploration processes starting from all the vertices of $V_{z_{1}},\ldots,V_{z_i}$.
We call $z_i\in \partial \B^*_{K/2}(x_0)$ \textbf{good} if none of its descendants in $\partial T_0$ (i.e.\ those vertices $y\in\partial T_0$ such that $d(x_0,y) = d(x_0,z_i) + d(z_i,y)$)  has been explored during the exploration processes corresponding to the sets $V_{z_1},\ldots, V_{z_{i-1}}$.  Otherwise, $z_i$ is called \textbf{bad}. Note that the event $\{z_i\text{ is bad}\}$ is $\F_{i-1}$ measurable. Finally, we denote by $\mathcal{D}_i$ the collection of vertices of $\G$ explored in the exploration process of the set $V_{z_i}$.
        }
\end{definition}

\begin{remark}\label{rem:stayinDi}
        \rm{We note that if the coupling between $X$ and $\til{X}$ starting from $x\in V_{z_i}$, where $z_i\in \partial \B_{K/2}^*(x_0)$, succeeds for $t$ steps, then $\til{X}_s \in \mathcal{D}_i$ for all $s\leq t$.
        }
\end{remark}

\begin{lemma}\label{lem:badandexplored}
 In  the  setup of Definition~\ref{def:coupling}, deterministically, $|\mathcal{D}_i| \le N= n\exp(-A \sqrt{\log n}/3)$ for all $i \in L$ (for all sufficiently large $n$). Moreover, there exists a positive constant~$C$ (independent of $T_0$) so that the number $\rm{Bad}$ of bad vertices $z$ satisfies 
        \[
        \prcond{{\rm{Bad}}\geq C \sqrt{\log n} }{\B_K^*(x_0)=T_0}{} \leq \frac{1}{n^2}.
        \]
\end{lemma}

\begin{proof}[\bf Proof]

Let $x\in V_z$ and let $\tree$ be the {\quasi} rooted at $x$ obtained during the exploration process of $\G$.
Let $k\geq 0$ and $S_k$ be the set of long range edges with one endpoint at level $k-1$ and the other one at level $k$ of $\tree$. Consider now
\[
\til{S}_k=\{e\in S_k: \trunc{e}{A}^c \text{ holds} \}.
\]
Recalling the definition of $\til{W}_\tree$ and of $\tau_{\ell(e)}$ from Definition~\ref{def:wtil}, we have 
\[
\sum_{e\in \til{S}_k} \exp(-\til{W}_\tree(e)) = \sum_{e\in \til{S}_k} \prcond{(X_{\tau^{\ell(e)-1}},X_{\tau^{\ell(e)}})=e}{\tree}{}\leq 1.
\]
Therefore, using the bound on $\til{W}_\tree$ from the truncation event, we obtain
\[
|\til{S}_k|\leq n\exp(-A \sqrt{\log n}),
\]
where $A$ is as in Lemma~\ref{lem:truncation}.
Since every long range edge we explore has a neighbourhood of radius $R$ around it, this means that when we reach distance $k$ from the root, we have revealed at most $\Delta^R |\til{S}_k|$ vertices, which is at most $n\exp(-A \sqrt{\log n}/2)$ for $n$ sufficiently large (recall that $R \asymp \log \log n$). Since the exploration process continues for $t\asymp \log n$ levels, the number of explored vertices is at most 
\begin{equation*}
N=n\exp\left(-A \sqrt{\log n }/3\right).
\end{equation*}
At every step of the exploration process the probability of intersecting a vertex of $\partial T_0$ is upper bounded by
\[
c_1 (\Delta)^{R(K+1)}/n\lesssim \Delta^{2C_1C_2(\log \log n)^2}/n,
\]
where $c_1$ is a positive constant.
We therefore obtain
\begin{align*}
\pr{{\rm{Bad}}>C\sqrt{\log n}} &\leq {N \choose C \sqrt{\log n}} \left(\frac{\Delta^{2C_1C_2(\log \log n)^2}}{n} \right)^{C \sqrt{\log n}} \\ &\leq N^{C \sqrt{\log n}} \left(\frac{\Delta^{2C_1C_2(\log \log n)^2}}{n} \right)^{C\sqrt{\log n}} \leq \frac{1}{n^{2}}
\end{align*}
by taking $C$ sufficiently large and using the definition of $N$.
\end{proof}

\begin{lemma}\label{lem:couplingsuccess}
        In the same setup as in Definition~\ref{def:coupling},
for all $\epsilon>0$, there exist $B$ (in the definition of $t$) and $A$ (in the definition of the truncation criterion, depending on $\epsilon$ and $B$) sufficiently large so that for all $n$ large enough, on the event $\{\B_K^*(x_0)=T_0\}$, for all $i$ and all descendants $x\in \partial \B_K^*(x_0)$ of $z_i$, the coupling of Definition~\ref{def:coupling} satisfies
\[
\prcond{\text{the coupling of $X$ and $\til{X}$ succeeds}}{\F_{i-1}}{x}  \geq \1(z_i\text{ is good}) \cdot (1-\epsilon).
\]     
\end{lemma}

\begin{proof}[\bf Proof]

We say that an overlap occurs at a vertex $y$ during the exploration process, when the ball $\B_{G_n}(y,R)$ (the ball of radius $R$ centred at $y$ w.r.t.\ $G_n$) revealed when exploring $y$ (i.e.\ when exploring some long range edge leading to $y$) intersects an already revealed $G_n^{*}\rball$. We say that the optimal coupling at a vertex has failed, if when revealing the endpoint of the long range edge coming out of it, the optimal coupling between the uniform distributions on the graph and the {\quasi} fails.  

We define $F$ to be the event that the walk $\til{X}$ crosses an edge of $\tree$ whose corresponding edge in $\G$ was truncated due to an overlap or because the optimal coupling failed. We first bound the probability of $F$. As in~\cite[Section~3.2]{BLPS}, we note that the event $F$ does not happen if the following occur: for each $i$ if the first time that $\til{X}$ reaches level $K+i$ there is no overlap and the optimal couplings succeed both at the current vertex and at all other vertices of $\tree$  at distance $2K$  from the walk at this time and in addition, if the walk never (by time $t$) revisits any vertex after visiting its depth $K$ descendants. (Indeed, if the walk never revisits any vertex after visiting its depth $K$ descendants, then each ball  visited by the walk by time $t$, say at level $\ell$, must be at distance at most $2K$ from the first ball to be visited at level  $\ell$.) This last event has failure probability at most $te^{-cK}$ by Lemma~\ref{lem:uniformdrift} and a union bound. Therefore, by choosing the constant $C_2$ in the definition of $K$ sufficiently large, this probability can be made $o(1)$. 
Since by  Lemma~\ref{lem:badandexplored} the total number of explored vertices in $\G$ is upper bounded by~$N$ (from Lemma \ref{lem:badandexplored}), the probability that the optimal coupling fails when the walk first visits level $K+i$ is at most\footnote{We are using the fact that the total variation distance between the uniform distributions on a set of size $n$ and on a subset of it of size $n-m$ is $m/n$.} $N/n$ and the probability that there is an overlap either there or at some vertex of the same level within distance $2K$ from it, is upper bounded by 
\[
\Delta^{R(2K+1)} \cdot \frac{\Delta^R N}{n-N}.
\]
By the union bound over all $t$ levels, we get that the probability that the event $F$ occurs is at most
\[t \cdot \Delta^{R(2K+1)} \cdot \frac{\Delta^R N}{n-N} + t \cdot \Delta^{R(2K+1)}\cdot \frac{N}{n} + t \cdot e^{-cK}= o(1),
\]
by choosing the constant $C$ in the definition of $K$ sufficiently large.

The coupling fails if the walk $\til{X}$ visits a truncated edge before time $t$ or if it visits a vertex $w\in \partial \B_{K/2}^*(x_0)$ with $w\neq z_i$.  But from Lemma~\ref{lem:truncation} (used to control the probability that the walk crosses an edge that got truncated due to the truncation criterion $\trunc{e}{A}$; edges that were truncated for other reasons were treated above), by choosing $A$ in the definition of the truncation criterion in terms of $\epsilon$ and $B$, we see that the first event has probability at most~$\epsilon/2$. The probability that $\til{X}$ visits a vertex $w\in \partial \B^*_{K/2}(x_0)$ with $w\neq z_i$ is at most $e^{-cK}$ for a positive constant $c$ by Lemma~\ref{lem:uniformdrift}, which is again $o(1)$ (recall that  $\til{X}$ starts from $x \in \partial T_0$ where $x$ is a descendant of $z_i$). 

Finally, another way for the coupling to fail is if the walk $\til{X}$ visits the boundary of a $\tree\rball$ before time $t$. Let $a$ be a centre of a $\tree\rball$ in $\tree$ and let $H_1$ be the event that $\til{X}$ ever visits the boundary of this $\tree\rball$ after having first visited its centre $a$. Writing $H_2$ for the event that $\til{X}$ visits this boundary after time $R-1$, we have 
\begin{align*}
        \prstart{H_1}{a}\leq \prstart{d(a,\til{X}_{R-1})\geq c_1 R, H_2 }{a}  + \prstart{d(a,\til{X}_{R-1})< c_1 R }{a}
        \lesssim e^{-c_2 R},
\end{align*}  
where the last inequality follows from Lemma~\ref{lem:uniformdrift} and Remark~\ref{rem:decayofdistance}. Since by time $t\asymp \log n$, the walk will visit at most $t$ different centres of balls, by taking a union bound and choosing the constant in the definition of $R$ sufficiently large, we get that the probability of this event happening is at most~$t e^{-c_2 R}=o(1)$.
\end{proof}
We denote by  $\trel(G)$ the \emph{absolute relaxation time}  of simple random walk on a finite graph $G$, defined as the inverse of the absolute spectral gap (it equals $+\infty$ if $G$ is bipartite or not connected). 
\begin{proposition}\label{pro:couplingandtv}
In the same setup as in Definition~\ref{def:coupling},
     for all $\epsilon>0$, there exist $B$ (in the definition of $t$), $A$ (in the definition of the truncation criterion) depending on $\epsilon$ and $B$ and a positive constant $\Gamma$ sufficiently large such that for all $n$ sufficiently large, on the event $\{\B^*_K(x_0)=T_0\}$, for all $i$ and all $x\in \partial \B_K^*(x_0)$ descendants of $z_i\in \partial \B_{K/2}^*(x_0)$, on the event $\{z_i \text{ is good}\}$ we have for all $s \ge 0$ that
        \[
        \prcond{d_x(t+s)<e^{-\tfrac{s}{\trel(\G)}}\cdot \frac{\sqrt{\Delta}}{1-\epsilon}  \exp(\Gamma \sqrt{\log n}) + \epsilon}{\F_{i-1}}{}\geq 1-2\epsilon,
        \]
        where $d_x(r)=\tv{\prcond{X_r\in \cdot}{\G}{x}-\pi}$ for every $r\in \N$.
\end{proposition}

\begin{proof}[\bf Proof]

We set $\ell= \log n/\mathfrak{h} - 2\nu  B\sqrt{\log n}$ and recall that $t=\log n/(\nu\mathfrak{h}) -B\sqrt{\log n}$. 

Let $\tree$ be the {\quasi} with root $x_0$ that we reveal during the exploration process of $\G$ starting from $x$ and which satisfies that $\B_K^*(x_0)=T_0$. 
Let $\xi$ be a loop erased random walk on $\tree$ started from $z_i$ as in Definition~\ref{def:lerw}, i.e.\ it is considered only when it crosses long range edges. As in the proof of Lemma~\ref{lem:badandexplored} we let $\til{S}_k$ be the set of long range edges of $\tree$ at distance $k$ from the root that do not satisfy the truncation criterion. For a constant $\Gamma$ to be determined we define
\[
\widehat{B} = \left\{e\in \til{S}_{\ell}: \ \prcond{\xi_{\ell}=e}{\tree}{} \leq \frac{1}{n} \exp\left( \Gamma\sqrt{\log n} \right)\right\}.
\]
Let $X$ be a simple random walk on $\G$ started from $x$ and let $\til{X}$ be a simple random walk on $\tree$ started from $x$ coupled with $X$ as in Definition~\ref{def:coupling}. 
Let $\til{\xi}$ be the loop erased random walk on $\tree$ obtained by erasing loops from $\til{X}$.
 Using Proposition~\ref{pro:entropy} and Lemma~\ref{lem:truncation} (for the event that $\til{\xi}_\ell \notin \til{S}_\ell$) we get that there exist $B$ in the definition of $t$ depending on $\epsilon$, $A$ in the definition of the truncation criterion (depending on $\epsilon$ and $B$) and $\Gamma$ (depending on $\epsilon$) sufficiently large such that 
\begin{align}\label{eq:bhat}
\prcond{\til{\xi}_{\ell}\in \widehat{B}}{\F_{i-1}}{}\geq 1-\epsilon^2/3.
\end{align}
We define the following events
\begin{align*}
A_1&=\left\{d(x_0,\til{X}_t)\geq \ell + \tfrac{\nu B}{2} \sqrt{\log n}  \right\}, \\  
A_2& = \left\{ \til{\xi}_{\ell} = (\text{loop erased trace of $(\til{X}_r)_{r\leq t}$ viewed on long range edges})_{\ell}    \right\},\\
A_3&= \left\{ \text{the coupling of $X$ and $\til{X}$ succeeds for $t$ steps} \right\}\cap \left\{ \til{\xi}_{\ell} \in \widehat{B}\right\}.
\end{align*}
Lemma~\ref{lem:speed} shows that for $B$ sufficiently large we have $\prcond{A_1}{\F_{i-1}}{} \geq 1-\epsilon^2/3$. Using Lemma~\ref{lem:uniformdrift} we get that for a positive constant $c$ we have 
\[
\prcond{A_2^c}{\F_{i-1}}{}\leq \exp\left(-c B\sqrt{\log n} \right) = o(1).
\]
Using Lemma~\ref{lem:couplingsuccess} and~\eqref{eq:bhat} for the probability of the event $A_3^c$ we deduce that if $S = \cap_{i=1}^{3}A_i$, then on the event $\{z_i \text{ is good}\}$ we have  $\prcond{S}{\F_{i-1}}{} \geq 1-\epsilon^2$.  Therefore setting
\[
\mathcal{G}=\{G: \prcond{S}{\G=G}{}\geq 1-\epsilon\}
\]
and using Markov's inequality and the tower property we obtain that on the event $\{z_i \text{ is good} \}$
\[
\prcond{\G\in \mathcal{G}}{\F_{i-1}}{} \geq  1-\epsilon.
\]
Let $s>0$ to be determined later. The above inequality now gives on the event $\{z_i \text{ is good} \}$
\begin{align}\label{eq:goodgraphs}
\begin{split}
        \mathbb{P}(&\tv{\prcond{X_{t+s}\in \cdot}{\G}{x} - \pi}  \\&= \1(\G \in\mathcal{G})\tv{\prcond{X_{t+s}\in \cdot}{\G}{x} - \pi}\mid \F_{i-1}) \geq 1-\epsilon.
\end{split}
\end{align}
We now have 
\begin{align}\label{eq:sumong}
        &\1(\G \in\mathcal{G})\tv{\prcond{X_{t+s}\in \cdot}{\G}{x} - \pi} \\&= \sum_{\substack{G\in \mathcal{G}}} \1(\G=G) \tv{\prcond{X_{t+s}\in \cdot}{\G=G}{x} - \pi},
\end{align}
and hence for each $G\in \mathcal{G}$, by conditioning on the event $S$ and using the definition of $\mathcal{G}$ we obtain
\begin{align}\label{eq:tvboundepsilon}
        \tv{\prcond{X_{t+s}\in \cdot}{\G=G}{x} - \pi}\leq \tv{\prcond{X_{t+s}\in \cdot}{S,\G=G}{x} - \pi} + \epsilon.
\end{align}
We next bound the first term appearing on the right hand side above. By the Poincar\'e inequality and the fact that conditional on $X_t$, the event $S$ is independent of $(X_{u})_{u\geq t}$ we have
\begin{align}\label{eq:poincare}
\begin{split}
        \tv{\prcond{X_{t+s}\in \cdot}{S,\G=G}{x} - \pi} &\leq \norm{\prcond{X_{t+s}\in \cdot}{S,\G=G}{x} - \pi}_2\\
        &\leq e^{-\tfrac{s}{\trel(G)}} \norm{\prcond{X_{t}\in \cdot}{S,\G=G}{x}- \pi}_{2},
        \end{split}
\end{align}
  
For every vertex $v\in \tree$ with $d(x_0,v)>\ell$ there is a unique ``ancestor edge'' $\phi(v)=(\phi(v)^-,\phi(v)^+)$ with $d(x_0,\phi(v)^-)=\ell$.
On the event $S$, the walk $X$ is coupled successfully with $\til{X}$ for $t$ steps, and hence we get
\begin{align*}
& \max_{v\in V_n}  \prcond{X_t=v}{\G=G, S}{x}  =\max_{v\in V_n} \frac{\E{\1(\G=G)\prcond{X_t=v, S}{\tree, \G}{x}}}{\prstart{\G=G, S}{x}} \\
 &= \max_{v\in \tree} \frac{\E{\1(\G=G)\prcond{\til{X}_t=v, S}{\tree, \G}{x}}}{\prstart{\G=G, S}{x}}
 \\ &\leq \max_{v\in \tree} \frac{\E{\1(\G=G)\prcond{\til{\xi}_\ell=\phi(v), S}{\tree, \G}{x}}}{\prstart{\G=G, S}{x}},
\end{align*}
where in the last inequality we used that on the event $S\subseteq A_2$ we have $\til{\xi}_{\ell} = \phi(v)$.
Using the definition of the set $\widehat{B}$ and of $S$ we have for all $v\in \tree$
\begin{align*}
&\E{\1(\G=G)\prcond{\til{\xi}_\ell=\phi(v), S}{\tree, \G}{x}} \\&=\estart{\1(\G=G)\1(\phi(v) \in \widehat{B})\prcond{\til{\xi}_{\ell}=\phi(v), S}{\tree, \G}{x}}{x} \\&\leq \frac{1}{n} \exp\left( \Gamma \sqrt{\log n} \right) \prstart{\G=G}{x}.
\end{align*}
Therefore, for $G\in \mathcal{G}$ this gives 
\begin{align*}
        \max_v \prcond{X_t=v}{\G=G, S}{x} &\leq \frac{1}{\prcond{S}{\G=G}{x}}\frac{1}{n} \exp\left( \Gamma \sqrt{\log n} \right) \\&\leq\frac{1}{(1-\epsilon)n} \exp\left( \Gamma \sqrt{\log n} \right), 
\end{align*}
where for the last inequality we used the definition of $\mathcal{G}$.  Using that $\pi$ is the degree biased distribution and that $G_n$ is a graph with maximum degree $\Delta$, we obtain that $\pi(v) \geq 1/(\Delta n)$ for all $v$.  Therefore, we obtain 
\begin{align*}
\norm{\prcond{X_{t}\in \cdot}{\G=G,S}{x}- \pi}_{2} \leq \frac{\sqrt{\Delta}}{1-\epsilon} \exp\left(\Gamma \sqrt{\log n}\right).
\end{align*}
 Plugging this into~\eqref{eq:poincare} and using~\eqref{eq:tvboundepsilon}, \eqref{eq:sumong} and~\eqref{eq:goodgraphs} we obtain on the event $\{z_i \text{ is good}\}$
\begin{align*}
\prcond{d_x(t+s) \leq e^{-\tfrac{s}{\trel(\G)}} \cdot  \frac{\sqrt{\Delta}}{1-\epsilon} \exp\left(\Gamma \sqrt{\log n}\right) + \epsilon}{\F_{i-1}}{} \geq 1-2\epsilon
\end{align*}
and this concludes the proof.
\end{proof}

\begin{lemma}\label{lem:exponentialdecayxi}
        There exists a positive constant $c$ so that for all quasi trees~$\tree$ rooted at $\rho$, all $\ell\in \N$ and all $x$ with $d(\rho,x)=\ell$, 
        if $X$ is a simple random walk started from $\rho$ and $\tau_\ell$ is the first hitting time of level $\ell$, then 
        \[
        \pr{X_{\tau_\ell}=x} \leq e^{-c\ell}.
        \]
\end{lemma}

\begin{proof}[\bf Proof]

Let $\xi$ be the loop erasure of the path $\{X_0,\ldots, X_{\tau_\ell}\}$. Let $e$ be the long range edge whose endpoint further from the root is $x$. Then 
\[
\pr{X_{\tau_\ell}=x} = \pr{\xi_\ell = e}.
\]
We let $e_1,\ldots, e_\ell=e$ be the sequence of long range edges leading from $\rho$ to $e$. We write $e_i=(x_i,y_i)$ with $d(\rho,x_i)<d(\rho,y_i)$. Using Lemma~\ref{lem:domainmarkov} we obtain
\begin{align}\label{eq:prodlerw}
        \pr{e\in \xi} =\prod_{i=1}^{\ell} \prstart{X^i_{L_i}=x_{i+1}}{y_i},
\end{align}
where $X^i$ is a simple random walk on $\tree(y_i)$ and $L_i$ is the last time before reaching level $\ell$ of $\tree$ that  $X^i$ is in the $\tree\rball$ centred at $y_i$ as in Lemma~\ref{lem:domainmarkov}. Writing $E_i$ for the first time $X^i$ leaves the $\tree\rball$ centred at $y_i$, we get
\begin{align*}
         \prstart{X^i_{L_i}=x_{i+1}}{y_i}  &= 1-\sum_{z\neq x_{i+1}} \prstart{X^i_{L_i}= z}{y_i}
       \\ & \leq 1-\sum_{z\neq x_{i+1}} \prstart{X_{E_i}^i=z, \tau^i_{z}=\infty}{y_i} \leq 1-c \prstart{X^i_{E_i}\neq x_{i+1}}{y_i}, 
\end{align*}
where $\tau^i_z$ stands for the first return time to $z$ after $E_i$ and in the last inequality we used Lemma~\ref{lem:uniformdrift}. Using the bounded degree assumption and that every connected component of $G_n$ contains at least $3$ vertices we deduce
\[
\prstart{X^i_{E_i}\neq x_{i+1}}{y_i}\geq c',
\]
where $c'$ is a positive constant. Therefore, this now implies that 
\begin{align*}
        \prstart{X^i_{L_i}=x_{i+1}}{y_i}\leq 1-cc',
\end{align*}
and hence plugging this into~\eqref{eq:prodlerw} finishes the proof.
\end{proof}

\begin{lemma}\label{lem:hitkroot}
There exists a positive constant $\alpha$, so that starting from any vertex the random walk will hit a $K$-root by time $\alpha K$ with probability $1-o(1)$ as $n\to\infty$.
\end{lemma}

\begin{proof}[\bf Proof]
Let $\beta\geq 3$ and suppose the random walk starts from $x$. We say that an overlap appears in $\B_{\beta K}^*(x)$, if there exist distinct vertices $y,z\in \B_{\beta K}^*(x)$  and a pair of long range edges $(y,y')$ and $(z,z')$ such that $\B_{G_n}(y',R)\cap \B_{G_n}(z',R)\neq \emptyset$. 
Let $M\leq \Delta^{\beta RK+1}$ be the number of points in $\B_{\beta K}^*(x)$. The probability that an overlap appears when exploring the long range edge attached to a vertex~$y$ is upper bounded by $\Delta^R M/n$. Therefore, the number of overlaps $I
$ in $\B_{\beta K}^*(x)$ is stochastically dominated by a binomial random variable ${\rm Bin}(M, \Delta^RM/n)$ and we have 
\begin{align*}
        \pr{I\geq 2}\leq {M\choose 2} \left(\frac{\Delta^R M}{n}\right)^2\leq \frac{\Delta^{2R} M^3}{n^2}  = n^{-2+o(1)},
\end{align*} 
using that $R+ K=O(\log \log n)$. Taking a union bound over all vertices of $G$ we get that with probability $1-o(1)$ the number of overlaps in the $\beta K$ ball around every vertex is at most~$1$. 
Therefore, if there is no overlap in $\B_{\beta K}^*(x)$, then the vertex $x$ is a $K$-root and we are done. If there is one overlap, then we consider the downward distance from the overlap at times that are multiples of $3$ exactly in the same way as in the proof of Lemma~\ref{lem:uniformdrift}. The rest of the proof of Lemma~\ref{lem:uniformdrift} follows verbatim, since having two centres in the overlap does not affect the proof that the drift is strictly positive.
\end{proof}

\begin{proof}[\bf Proof of Theorem~\ref{thm:1}]

Recall the definition of $t$ from~\eqref{eq:defoft}
\[
t=\frac{\log n}{\nu\mathfrak{h}} - B \sqrt{\log n},
\]
where $B$ is a positive constant to be chosen later.
We first prove the upper bound on the mixing time. 
Let $s= \trel(\G)\left[ \Gamma \sqrt{\log n} +\log \frac{\sqrt{\Delta}}{\eps(1-\eps)}\right]$, where $\Gamma$ is as in Proposition~\ref{pro:couplingandtv} so that
\[
e^{-\tfrac{s}{\trel(\G)}} \cdot\frac{\sqrt{\Delta}}{1-\epsilon} \exp\left(\Gamma\sqrt{\log n} \right) =\epsilon.
\] 
We claim that it suffices to prove that w.h.p.
\begin{align}\label{eq:whptmixtrel}
\tmixx{\G}{7\epsilon} \leq t+s + (\alpha +c)K,
\end{align}
where  $\alpha$ is as in Lemma~\ref{lem:hitkroot} and $c$ is a positive constant to be determined later. Indeed, one can then easily finish the proof, since by Proposition \ref{pro:expander} (whose proof is deferred to Section~\ref{s:expander}) there exists some constant $\hat \alpha>0$ such that w.h.p. its absolute relaxation time $ \trel(\G) $ is at most $1/\hat \alpha $. Hence this together with~\eqref{eq:whptmixtrel} gives the desired upper bound on $\tmixx{\G}{\epsilon}$.

We now prove~\eqref{eq:whptmixtrel}.
By Lemma~\ref{lem:hitkroot}, the strong Markov property (applied to the first hitting time of a $K$-root) and the fact that the total variation distance from stationarity is non-increasing, we have w.h.p.
\begin{align*}
        &\max_x \tv{\prcond{X_{t+s+(\alpha+c) K}\in \cdot}{\G}{x} - \pi}  \\&\leq \max_{x_0: \ K\text{-root}} \tv{\prcond{X_{t+s+c K}\in \cdot}{\G}{x_0} - \pi}  + o(1).
\end{align*}
From now on we fix $x_0$ a $K$-root of $\G$ and set
\[
V =\left \{x\in \partial \B^*_K(x_0): d_x(s+t) \geq 2 \epsilon \right\}.
\]
(Note that  this is a random set which depends on $G_n^*$.) Letting $\tau_K$ be the first hitting time of $\partial \B_K^*(x_0)$, we claim that it suffices to prove that 
\begin{align}\label{eq:hitV}
\pr{\prcond{X_{\tau_K}\in V}{\G}{x_0}>4\epsilon, \ x_0 \text{ is a $K$-root}}  \leq         \frac{1}{n^2}.
\end{align}
Indeed, this will imply that 
\begin{align*}
\pr{\exists  \ \text{a $K$-root} \ x_0:  \prcond{X_{\tau_K}\in V}{\G}{x_0}>4\epsilon} \leq \frac{1}{n}
\end{align*}
and then the proof will follow easily, since using the strong Markov property and the non-decreasing property of the total variation distance from stationarity we get for any $K$-root $x_0$ and any $T_0$, on the event $\{\B^*_K(x_0)=T_0\}$
\begin{align*}
&        \tv{\prcond{X_{t+s+c K}=\cdot }{\G}{x_0} - \pi} \leq \prcond{\tau_{\partial T_0}>cK}{\G}{x_0}\\
         & + \sum_{z\in \partial T_0} \prcond{X_{\tau_{\partial T_0}}=z}{\G}{x_0}\cdot \tv{\prcond{X_{t+s}=\cdot }{\G}{z}-\pi} \\
        & \leq o(1) + \prcond{X_{\tau_{K}}\in V}{\G}{x_0} + 2\epsilon,
\end{align*}
where the second inequality follows from taking $c$ sufficiently large and using Lemma~\ref{lem:uniformdrift} for the first term and the definition of the set $V$ for the bound on the sum.

We now prove~\eqref{eq:hitV}.
We write $h(x) = \prcond{X_{\tau_K}=x}{\G}{x_0}$ to simplify notation. 
As in Definition~\ref{def:coupling} let $V_{z_i}$ be the set of descendants of $z_i$ in $\partial \B_K^*(x_0)$ and $L=|\partial \B_{K/2}^*(x_0)|$. Recalling from the same definition the notions of bad and good vertices on $\partial \B_{K/2}^*(x_0)$ we get
\begin{align}\label{eq:decomposh}
\begin{split}
        h(V)  = \sum_{i=1}^{L} h( V\cap V_{z_i}) = \sum_{i=1}^{L}h(V\cap V_{z_i})\1(z_i \text{ is good}) \\+ \sum_{i=1}^{N}h(V\cap V_{z_i})\1(z_i \text{ is bad}).
        \end{split}
\end{align}
Let $\xi$ denote the loop erasure of $(X_t)_{t\leq \tau_K}$. Then by Lemma~\ref{lem:uniformdrift} on the event $\{\B_K^*(x_0)=T_0\}$ we have 
 \begin{align*}
        \prcond{X_{\tau_K}\in V_{z_i}}{\G}{x_0} =\prcond{\xi_{K-1} \in V_{z_i}}{\G}{x_0}= \prcond{\xi_{K/2-1}=z_i}{\G}{x_0} \leq e^{-cK/2},
\end{align*}
 where the last inequality follows from Lemma~\ref{lem:exponentialdecayxi} and $c$ is a positive constant.
Choosing the constant $C_2$ in the definition of $K$ sufficiently large we get that for all $i$
\begin{align}\label{eq:boundonhvzi}
h(V_{z_i}) \leq \frac{1}{(\log n)^2}.
\end{align}
Using also Lemma~\ref{lem:badandexplored} we now obtain
\begin{align}\label{eq:nobad}
        \prcond{\sum_{i=1}^{L}h(V\cap V_{z_i})\1(z_i \text{ is bad})\leq \frac{C}{(\log n)^{3/2}}}{ \B_K^*(x_0)=T_0}{} \geq 1-\frac{1}{n^2},
\end{align} 
where $C$ is the constant from Lemma~\ref{lem:badandexplored}.
We now turn to the first term on the right hand-side of~\eqref{eq:decomposh}. We start by writing each term of the sum as
\begin{align*}
        h(V\cap V_{z_i})\1(z_i \text{ is good}) = \sum_{x\in V_{z_i}} h(x) \1(d_x(t+s)\geq 2\epsilon) \1(z_i \text{ is good}).
\end{align*}
So by Proposition~\ref{pro:couplingandtv} and our choice of $s$ we have

\begin{align}
\label{e:2epshV}
        \econd{h(V\cap V_{z_i})\1(z_i \text{ is good}) }{\F_{i-1}}{}\1(\B_K^*(x_0)=T_0)\leq 2\epsilon h(V_{z_i}).
\end{align}
Writing $R_i$ for the random variable appearing in the conditional expectation above, we consider the martingale defined conditionally on $\B_K^*(x_0)=T_0$ via $M_0=0$ and for $1 \le k\le L$
\[
M_k=\sum_{i=1}^{k} \left( R_i  - \econd{R_i}{\F_{i-1}} \right).
\]
Applying then the Azuma-Hoeffding inequality to this martingale we obtain that for a positive constant $c$ 
\begin{align*}
        &\prcond{\sum_{i=1}^{L}h(V\cap V_{z_i})\1(z_i \text{ is good}) >3\epsilon}{\B_K^*(x_0)=T_0}{} \\& \le \prcond{M_L > \eps }{\B_K^*(x_0)=T_0}{} \leq \exp\left(-\frac{c\epsilon^2}{\sum_{i=1}^{L}(h(V_{z_i}))^2} \right)
        \leq \exp\left(-c\epsilon^2 (\log n)^2 \right),
\end{align*}
where for the first inequality we used \eqref{e:2epshV} and for the last inequality we used that 
        \[
        \sum_{i=1}^{L}(h(V_{z_i}))^2\leq \frac{1}{(\log n)^2} \sum_{i=1}^{L} h(V_{z_i}) =\frac{1}{(\log n)^2},
        \]
which follows from~\eqref{eq:boundonhvzi} (we also used that  $|M_{i}-M_{i-1}| \le h(V_{z_i})$ for all $i \le L$ and conditioned on $\B_K^*(x_0)=T_0$, we have that $h(V_{z_i})$ is deterministic). This shows that $h(V) \le 3\eps+o(1)$ with probability at least $1-2/n^2$, thus concluding the proof of the upper bound on the mixing time.

We now prove the lower bound. We employ the same notation as in the proof of the upper bound.
Suppose  the walk starts from a vertex $x_0$ which is a $K$-root. Recall that  $\mathcal{D}_i$ is collection of vertices of $\G$ explored in the exploration process of the set $V_{z_i}$. On the event that $x_0$ is a $K$-root, set
\[
V' =\left \{z_{i}\in \partial \B^*_{K/2}(x_0): \mathbb{P}_{z_i}( T_{(\mathcal{D}_i\cup\B_K^*(x_0))^{c}} \le t,   X_{\tau_K} \in V_{z_i}\mid \G) \geq 2 \epsilon \right\}
\]
(recall that $\tau_K$ is the first hitting time of $\partial \B_K^*(x_0)$). Let $\mathcal{D}(x_{0})=(\cup_{i=1}^{L}\mathcal{D}_i)\cup\B_K^*(x_0) $ and $\widehat{V}=\cup_{z \in V'}V_z$. By the strong Markov property, as well as the fact that if $x_0$ is a $K$-root and $z_{i} \in \partial \B_{K/2}^*(x_0)$, then starting from $x_0$ a walk must visit $z_i$ prior to time $\tau_K$ in order to have that $X_{\tau_K} \in V_{z_i}$,  on the event that $x_0$ is a $K$-root, we get that  
\begin{align*}
  \mathbb{P}_{x_0}( T_{\mathcal{D}(x_{0})^{c}} \le t,X_{\tau_K} \in V_{z_i}   \mid \G)   
\le \mathbb{P}_{z_i}( T_{(\mathcal{D}_i\cup\B_K^*(x_0))^{c}} \le t,X_{\tau_K} \in V_{z_i}   \mid \G).     \end{align*}
Recalling that $h(x)=h_{x_{0}}^{\G} (x)= \prcond{X_{\tau_K}=x}{\G}{x_0}$, and summing over $i \in [L] $ we see that
\[\mathbb{P}_{x_0}( T_{\mathcal{D}(x_{0})^{c}} \le t\mid \G) \le 2\varepsilon+\prcond{X_{\tau_K}\in\widehat{V} }{\G}{x_0}=2\varepsilon +h_{x_{0}}^{\G} (\widehat{V}) \]

 (on the event that $x_0$ is a $K$-root). We claim that it suffices to prove that 
\begin{align}\label{eq:hitVlower}
\pr{h_{x_{0}}^{\G} (\widehat{V})>4\epsilon, \ x_0 \text{ is a $K$-root}}  \leq         \frac{1}{n^2}.
\end{align}
Indeed, by a union bound, this will imply that 
\begin{align*}
\pr{\exists  \ \text{a $K$-root} \ x_0:  \prcond{X_{t}\notin \mathcal{D}(x_{0}) }{\G}{x_0}>6\epsilon} \leq \frac{1}{n}.
\end{align*}
The proof of the lower bound could then be concluded by noting that (i) by Lemma \ref{lem:hitkroot} $K$-roots exist w.h.p., and (ii) by Lemma \ref{lem:badandexplored}  $|\mathcal{D}(x_{0})|=o(n)$ and hence by the bounded degree assumption $\pi(\mathcal{D}(x_{0}))=o(1)$.  Indeed, we would get that with probability $1-o(1)$ there exists a $K$-root $x_0$ so that
\[
\tv{\prcond{X_t\in \cdot }{\G}{x_0} - \pi} \geq\prcond{X_t\in  \mathcal{D}(x_0) }{\G}{x_0} - \pi( \mathcal{D}(x_0)) \geq 1-6\epsilon -o(1).
\]
So it remains to prove \eqref{eq:hitVlower}. Using Remark~\ref{rem:stayinDi} together with Lemma~\ref{lem:couplingsuccess} we obtain 
\begin{align*}
        \econd{h( V_{z_i})\1(z_i \in V'  \text{ and is good}) }{\F_{i-1}}{}\1(\B_K^*(x_0)=T_0)\leq 2\epsilon h(V_{z_i}).
\end{align*}
Writing $R_i'$ for the random variable appearing in the conditional expectation above, we consider the martingale defined conditional on $\B_K^*(x_0)=T_0$ as  $M_0'=0$ and  $M_k'=\sum_{i=1}^{k} \left( R_i'  - \econd{R_i'}{\F_{i-1}} \right)$ for $1 \le k\le L$.
Applying the Azuma-Hoeffding inequality to this martingale we obtain exactly as in the proof of the upper bound that for some positive constant $c$ we have that
\begin{align*}
        &\prcond{\sum_{i=1}^{L}h( V_{z_i})\1(z_i \in V'  \text{ and is good})  >3\epsilon}{\B_K^*(x_0)=T_0}{} \\& \le \prcond{M_L' > \eps }{\B_K^*(x_0)=T_0}{} \leq \exp\left(-\frac{c\epsilon^2}{\sum_{i=1}^{L}(h(V_{z_i}))^2} \right)
        \leq \exp\left(-c\epsilon^2 (\log n)^2 \right).
\end{align*}
Using  $h(\widehat{V}) \le \sum_{i=1}^{L}h( V_{z_i})\1(z_i \in V'  \text{ and is good})+\sum_{i=1}^{L}h( V_{z_i})\1(z_i  \text{ is bad})$ (analogously to \eqref{eq:decomposh}) together with \eqref{eq:nobad} concludes the proof of \eqref{eq:hitVlower} and thus of the lower bound.
\end{proof}
\begin{remark}
\label{r:beststartingpoint} It is not hard to show that w.h.p.\ $\G$ satisfies  for some constant $\beta \ge 3 $ that for all  $x$ we have that if $W_x$ is the collection of $K$-roots at distance at most $\beta K$ from $x$ then \[\mathbb{P}_x(X_t \in \cup_{w \in W_x}\mathcal{D}(w) \mid \G ) \ge 1-7 \eps. \]
This means that w.h.p.\ $\min_x d_x(t) \ge 1-8 \eps$ (as deterministically $\max_x|\cup_{w \in W_x}\mathcal{D}(w)|=o(n) $).
\end{remark}
\subsection{A family of examples demonstrating the necessity of the degree assumption}\label{s:example}

Consider a random $d_n$-regular   graph $G_n'$ of size $n$, where $  \frac{\log {n}}{\log \log {n}} \lesssim d_n = n^{o(1)}$. Now obtain a new graph $G_n$ by adding a clique of size $d_{n}$ and connecting a single vertex of the clique to one vertex of $G_n$ by an edge. One can verify that the mixing time of $G_n^*$ is of order $d_n$ and that there is no cutoff since starting from the clique, the time it takes the walk to first exit the clique stochastically dominates the Geometric distribution with mean $d_n/2$. To see this, observe that the walk on $G_n^*$ exits the clique in $O(d_n)$ steps, and is unlikely to return to it in the following $2\log_{d_n}{n}=O(d_n) $ steps. Hence for the following $2\log_{d_n}n$ steps the walk can be coupled with that on the induced graph (w.r.t.\ $G_n^*$) on the vertices of $G_n'$. 
This graph is similar to a random graph with a given degree sequence in which $n-d_n$ vertices have degree $d_n+1$ while the rest have degree $d_n$ (we write ``similar" as it need  not be a simple graph).  In fact, the walk is unlikely to visit any degree $d_n$ vertices during these $2\log_{d_n}n$ steps (other than when just leaving the clique) or vertices belonging to cycles of size $2$ by this time, and thus the argument from \cite[Corollary 4]{LS:cutoff-random-regular} (asserting the mixing time of a random $d_{n}+1 \gg 1$ regular graph on $M$ vertices  is $(1+o(1))\log_{d_n}M$) applies here.

\section{Expander}
\label{s:expander}
In this section we prove Theorem~\ref{t:expander} which is a more quantified version of
Proposition \ref{pro:expander}. 

We denote the second largest eigenvalue of a matrix $P$ by $\lambda_2(P)$ and its smallest eigenvalue by $\lambda_{\min}(P)$. When $P$ is the transition matrix of simple random walk on a graph $G$ we write  $\lambda_2(G)$ and $\lambda_{\min}(G)$ for   $\lambda_2(P)$ and $\lambda_{\min}(P)$, respectively.    

\begin{theorem}
\label{t:expander}
Let $G=(V,E)$ be an $n$-vertex graph of maximal degree $\Delta$. Assume that all connected components of $G$ are of size at least $3$. Let $G^*$ be the graph obtained from $G$ by picking a random perfect matching of $V$ (if $n$ is odd, one vertex remains unmatched) and adding edges between matched vertices. Then there exists some $\alpha=\alpha(\Delta) \in (0,1)$ such that
\begin{equation*}
\pr{1-\lambda_2(G^{*} )\leq \alpha} \lesssim n^{- \alpha} \quad  \text{ and } \quad 
\pr{1+\lambda_{\min}(G^{*} )\leq \alpha} \lesssim n^{- \alpha}. 
\end{equation*}
\end{theorem}

In the proof of the first inequality above we are going to use the following result from~\cite{madrasrandall}. For a similar result see also~\cite{jerrumetal}.

\begin{theorem}[{\cite[Theorem~1.1]{madrasrandall}}]\label{thm:madrasrandall}
Let $X$ be a reversible Markov chain with transition matrix $P$, invariant distribution $\pi$ and spectral gap $\gamma$. Let $V_1,\ldots,V_M$ be a partition of $V$ and let $P_i$ be the transition matrix on $V_i$ with off-diagonal transitions $P_i(x,y)=P(x,y)$ for all $x\neq y \in V_i$ and $P_i(x,x)=1-\sum_{z \in V_i \setminus \{x \}}P(x,z)$. Denote its spectral gap  by $\gamma(P_i)$ and let $\gamma_{*}:=\min_{i \in [M]}\gamma(P_i)$. Let $\widehat P$ be a Markov chain on $[M]$ with transition probabilities given by 
\begin{align}\label{eq:defphat}
\widehat P(i,j)=\prcond{X_1 \in V_j}{ X_0 \in V_{i}}{\pi}=\sum_{x \in V_i}\frac{\pi(x)}{\pi(V_i)}P(x,V_j).
\end{align}
and spectral gap given by $\widehat \gamma$. Then 
\begin{equation}
\label{e:decom}
\gamma \ge   \widehat\gamma \gamma_{*}. 
\end{equation}
\end{theorem}

We now recall an extremal characterization of $\lambda_{\min}(P)$ which will be used in the proof of the inequality for $\lambda_{\rm{min}}$ in the proof of Theorem~\ref{t:expander}.

\begin{theorem}[{\cite[Theorems~3.1 and 3.2]{smallesteigenvalue}}]
Let $P$ be a reversible transition matrix on a finite state space $V$ with invariant distribution $\pi$ and $Q(D,F)= \sum_{d\in D, f\in F}\pi(d)P(d,f)$ for $D, F\subseteq V$. For a set $S\subseteq V$ let
\begin{align*}
\zeta(S)=\min_{A,B: \, A \cup B=S,\, A \cap B=\emptyset}\zeta(S,A,B), \text{ where } \\ \zeta(S,A,B)=\frac{Q(A,A)+Q(B,B)+Q(S,S^c)}{\pi(S)}
\end{align*}
and define $\zeta_{*}=\min_{\emptyset \neq S \subseteq V}\zeta(S)$. Then we have 
\begin{equation}
\label{e:smallestev}
1-\sqrt{1-\zeta_{*}^{2}} \le 1+\lambda_{\min}(P) \le 4 \zeta_{*}.
\end{equation}
\end{theorem}

We also recall Cheeger's inequality (see for instance~\cite[Ch.\ 13]{LevPerWil})
\begin{equation}
\label{e:cheeger}
1-\sqrt{1-\Phi_{*}^{2}} \le 1-\lambda_{2}(P) \le 2 \Phi_{*}, \quad \text{where} \quad  \Phi_{*}=\min_{S:\, 0<\pi(S) \le 1/2}\Phi(S),
\end{equation}
and $\Phi(S)=\frac{Q(S,S^c)}{\pi(S)}$.

The next two lemmas will be used in the proof of Theorem~\ref{t:expander}. We defer their proofs to the end of this section.

\begin{lemma}
\label{l:partition}
Assume that the minimal size of a connected component of $G=(V,E)$ is at least $L$. Then there exists a partition  $V_1,\ldots,V_{M}$  of $V$ such that  for all $i \in [M]$  the induced graph on $V_i$ is connected and  $L \le |V_i| < L^2\Delta$, where $\Delta$ is the maximal degree in $G$.
\end{lemma}

\begin{lemma}\label{lem:probforamatching}
Let $\Omega$ be a set on $n$ vertices and let  $A\subseteq \Omega$ be a set satisfying $|A|=\alpha n$ with $\alpha\in (1/2 - \delta, 1/2+\delta)$ and $\delta<1/4$. Pick a perfect matching on $\Omega$ uniformly at random. We then have 
\begin{align*}
\pr{\exists \text{ less than $\delta n$ edges of perfect matching joining pairs of vertices of $A$}} \\ \leq 2^{-n\cdot \left( \frac{1}{2} -C(\delta)\right)},
\end{align*}
where $C(\delta)$ is a constant depending on $\delta$ satisfying $\lim_{\delta\to 0}C(\delta)=0$.
\end{lemma}

\begin{proof}[\bf Proof of Theorem~\ref{t:expander}]

 Taking $L=3$ we can apply Lemma~\ref{l:partition} to get a partition $V_1,\ldots, V_M$ of connected components of $V$ (with respect to the graph structure of $G$) such that $3\leq |V_i|\leq 9\Delta$ for all $i \in [M]$. We start by proving that there exists $\alpha>0$ such that 
 \begin{align}\label{e:la2}
 \pr{1-\lambda_2(G^{*} )\leq \alpha} \lesssim n^{- \alpha}.
 \end{align}
Let $X$ be a simple random walk on $G^*$ and let $\widehat{P}$ be the transition matrix as in~\eqref{eq:defphat}. Let also $P_i$ be the transition matrix on $V_i$ and $\widehat{\gamma}$ and $\gamma_*$ be as in the statement of Theorem~\ref{thm:madrasrandall}.

Consider the multigraph $H=([M], \til{E})$, in which the number of edges joining vertices $i$ and $j$ is equal to the number of edges of the perfect matching between $V_i$ and $V_j$ (and the number of loops of vertex $i$ is equal to the number of pairs of vertices of $V_i$ that are matched to each other). Then this multigraph is distributed as the configuration model on $[M]$ where vertex $i$ has degree $|V_i|$. Let~$K$ be the transition matrix of simple random walk on $H$. We are going to compare $\widehat{P}$ to $K$ as well as their invariant distributions and then using standard comparison techniques we will be able to compare their spectral gaps. Let $E^*(V_i,V_j)$ be the number of edges of $G^*$ that join vertices of $V_i$ and $V_j$ and let $E(V_i, V_j)$ be the number of edges of $G$ joining vertices of $V_i$ to vertices of $V_j$. Using the definition of $\widehat{P}$ and $K$ we get for $i\neq j$
\begin{align*}
\widehat{P}(i,j) = \frac{E^*(V_i,V_j)}{\sum_{v\in V_i} (\deg(v) +1)} \quad \text{and} \quad K(i,j) = \frac{E^*(V_i,V_j) - E(V_i,V_j)}{|V_i|},
\end{align*}
and hence writing $\pi_{\widehat{P}}$ and $\pi_K$ for the corresponding invariant distributions we get for all $i\in [M]$
\begin{align*}
\pi_{\widehat{P}}(i) = \frac{\sum_{v\in V_i} (\deg(v) +1)}{2|E| + n} \quad \text{and} \quad \pi_K(i) = \frac{|V_i|}{n}.
\end{align*}
Therefore, we obtain for all $i,j\in [M]$
\[
\widehat{P}(i,j) \geq \frac{1}{\Delta+1} K(i,j) \quad \text{and} \quad \pi_{\widehat{P}}(i) \geq \frac{1}{\Delta+1}\pi_K(i),
\]
and hence using the extremal characterisation of the spectral gap in terms of the Dirichlet form  (see for instance~\cite[Ch.\ 13]{LevPerWil}) we obtain
\begin{align}\label{eq:widehatgamma}
\widehat{\gamma} \ge \frac{1-\lambda_2(K)}{(\Delta+1)^2}.
\end{align}
For the random walk on the configuration model it is known (see for instance~\cite[p.\ 149-150]{anatomy}) that for some $\alpha>0$
\[
\pr{1-\lambda_2(K) <\alpha} \lesssim n^{-\alpha}.
\]
Using this, the inequality $\gamma_*\gtrsim  (\max_{i}|V_i|)^{-3}$ (see for instance~\cite[Ch.\ 6]{AF}), Theorem~\ref{thm:madrasrandall}  and~\eqref{eq:widehatgamma} finishes the proof of~\eqref{e:la2}.

We now prove that there exists $\alpha>0$ so that 
\begin{equation}
\label{e:lmin}
\pr{1+\lambda_{\min}(G^{*} )\leq \alpha} \lesssim n^{- \alpha}. 
\end{equation}
We first argue that it suffices to consider only sets $S$ of size at least $(1-\delta)n$, for some constant~$\delta>0$, by showing that otherwise $\zeta(S)$ is bounded away from 0.

By \eqref{e:la2} and \eqref{e:cheeger} there exists $\beta=\beta(\Delta)>0$ such that 
\begin{align}\label{eq:phistar}
\pr{\Phi_{*} \le \beta} \leq n^{-\beta}.
\end{align}
Let $\delta>0$ to be determined later. 
On the event $\{\Phi_{*} > \beta\}$,  using $Q(S,S^c)=Q(S^c,S)$ and \eqref{e:smallestev}  we see that every $S \neq \emptyset$ with $|S| \leq (1-\delta)n$ satisfies $\zeta(S) \ge \Phi(S) \gtrsim  \delta \beta$.  Defining 
\[
\mathcal{S} = \{ S\subseteq V: |S|\geq (1-\delta)n\} \quad \text{ and } \quad  \xi = \min_{S\in \mathcal{S}} \zeta(S)
\]  
we see that it suffices to show that for some constant $c=c(\Delta)>0$ we have 
\[
\pr{\xi \le c \delta} \lesssim n^{-\delta}.
\]
Let $S \in \mathcal{S}$ and $A,B$ be a partition of $S$. If there are at least $\delta n$ edges (either of the base graph $G$, or of the random perfect matching) connecting pairs of vertices of $A$ or pairs of vertices of $B$, then for some $c(\delta,\Delta)>0$ we  have that
\begin{equation}
\label{e:cdeltaDelta}
\zeta(S,A,B)  \geq Q(A,A)+Q(B,B)\ge c(\delta,\Delta).
\end{equation}
In particular, \eqref{e:cdeltaDelta} holds if $|A|\geq n(1+\delta)/2$ or $|B|\geq n(1+\delta)/2$, since then by simple counting, there must exist at least $\delta n/2$ edges between pairs of vertices of $A$ or $B$.

So from now on we restrict to partitions $(A,B)$ of $S$ which satisfy $|A|, |B| \in ((1/2-\delta)n ,(1/2+\delta)n)$.  Recall the definition of the partition $V_1,\ldots, V_M$ of $V$. For each $i$ let $U_i(1)$ and $U_i(2)$ be the partition of $V_i$ such that  
\[
Q(U_i(1), U_i(1)) + Q(U_i(2), U_i(2)) = \min_{(U,W) \text{ partition of } V_i} (Q(U,U) + Q(W,W)).
\]
Since there can be at most one partition (up to relabeling of the two sets) for which the sum above is equal to $0$, which happens in the case of an induced bipartite graph, it follows that for every other partition~$U,W$ of $V_i$ we have 
\begin{align}\label{eq:otherpartition}
Q(U,U)+ Q(W,W) \geq \frac{1}{2|E|+n}.
\end{align}

We call a partition $(A,B)$ of $S$ (satisfying $|A|, |B| \in ((1/2-\delta)n ,(1/2+\delta)n)$) \emph{good} if
the number of $i\in [M]$ for which
\begin{align}\label{eq:conditionongood}
(U_i(1) \subseteq A\cup S^c \  \text{ and } \  U_i(2) \subseteq B)  \quad \text{ or } \quad (U_i(2) \subseteq A\cup S^c \  \text{ and } \  U_i(1) \subseteq B)
\end{align}
is less than $M-8\Delta \delta n$.

Otherwise, $(A,B)$ is called \emph{bad}. Writing $A_i=A\cap V_i$ and $B_i=B\cap V_i$ and using that 
\begin{align*}
Q(A_i,A_i) =Q((A\cup S^c)\cap V_i, (A\cup S^c)\cap V_i)  
&- Q(S^c\cap V_i, S^c\cap V_i) \\&- 2Q(A_i, S^c\cap V_i),
\end{align*}
we get that if $(A,B)$ is a good partition, then 
\begin{align*}
\zeta(S,A,B)\geq Q(A,A)+Q(B,B) &\geq \sum_{i=1}^{M} (Q(A_i,A_i) + Q(B_i,B_i))\\ &\geq  \frac{8\Delta \delta n}{2|E|+n} - 3\pi(S^c) \gtrsim \delta.
\end{align*}
Note that for the third inequality we used  that for the indices $i$ for which~\eqref{eq:conditionongood} does not hold, the pair $((A\cup S^c)\cap V_i, B_i)$ is a partition of $V_i$ different to $(U_i(1), U_i(2))$, and hence for these indices we can apply~\eqref{eq:otherpartition}.  For each partition $(A,B)$ we define the event 
\begin{align*}
L(A,B) = \{ \text{number of edges of matching between pairs of}\\  \text{vertices of $A$ or $B$ is less than } \delta n\}.  
\end{align*}
We now deduce the following bound
\begin{align*}
\pr{\xi\leq c\delta} &\leq \pr{\exists S\in \mathcal{S}, \exists \text{ bad partition } (A,B) \text{ of } S: L(A,B)}
\\ &\leq |\mathcal{S}|\cdot  \max_{S\in \mathcal{S}}|\{(A,B) \text{ bad partition of } S\}|\cdot  \pr{L(A,B)}.
\end{align*}
We now claim that the number of bad partitions of $S$ is upper bounded by $2^{n/3 + C'(\delta) n}$ for some constant $C'(\delta)>0$ with $C'(\delta)\to 0$ as $\delta\to 0$. Indeed, the sets $A$ and $B$ of the partition are completely determined by the sets $((A_i, B_i))_{i\leq M}$. Now for each $i$ such that~\eqref{eq:conditionongood} holds, the set~$A_i$ must belong to the set $\{U_i(1)\cap S, U_i(2)\cap S\}$. Since $|V_i\cap S|\leq |V_i|\leq 9\Delta$, for the indices~$i$ such that~\eqref{eq:conditionongood} holds we can pick $A_i$ in at most  $2^{9\Delta}$ different ways. Therefore we obtain for all $S\in \mathcal{S}$
\begin{align*}
 |\{(A,B) \text{ bad partition of } S\}| \leq \sum_{k\leq 8\Delta \delta n} {M\choose k} 2^{M-k} 2^{9\Delta k} \leq 2^{n/3} \cdot 2^{C'(\delta) n},
\end{align*}
where $C'(\delta)$ is a constant as claimed above and where we also used that $M\leq n/3$, since $|V_i|\geq 3$ for all~$i$. So we can now conclude using also Lemma~\ref{l:partition}
\begin{align*}
\pr{\xi\leq c\delta} \leq {n\choose \delta n} \cdot  2^{n/3} \cdot 2^{C'(\delta) n} \cdot 2^{-n \cdot (\tfrac{1}{2} - C(\delta))} \leq n^{-\alpha}
\end{align*}
for some $\alpha=\alpha(\Delta)>0$, where the last inequality follows from taking $\delta$ sufficiently small. This now concludes the proof.
\end{proof}

\begin{proof}[\bf Proof of Lemma~\ref{l:partition}]
We define the sets of the partition inductively, using a greedy procedure. 
After defining $V_1,\ldots,V_j$ such that
\begin{itemize}
\item for all $i \in [j]$ the induced graph on $V_i$ is connected and $L \le |V_i| \le L^2\Delta$ and 
\item all connected components of  the induced graph on $B:=V \setminus \cup_{i\in [j]}V_i$ are of size at least $L$,
\end{itemize}
 we proceed to define $V_{j+1}$  
such that the same hold w.r.t.\  $V_1,\ldots,V_{j+1}$. We  pick an arbitrary connected set $A \subset B $ of size $L$. If all connected components of the induced graph on $B \setminus A$ are of size at least $L$ then we set $V_{j+1}=A$. Otherwise, we set $V_{j+1}$ to be the union of $A$ with all the connected components of the induced graph on $B \setminus A$ of size less than $L$. By the induction hypothesis, each such connected component must be adjacent to $A$, and so indeed $|V_j| < L^2\Delta$ as desired. This concludes the induction step. Note that the above description of $V_{j+1}$ can also be used to define $V_1$. 
\end{proof}

\begin{proof}[\bf Proof of Lemma~\ref{lem:probforamatching}]

Write $m=\alpha n$. For the probability in question we then have 
\begin{align*}
&\pr{\exists \text{ less than $\delta n$ edges joining pairs of vertices of $A$}} \\ &= \sum_{m-\tfrac{n}{2}\leq i\leq \delta n}  {m \choose 2i} \cdot \frac{(n-m)!}{(n-2m+2i)!} \cdot \frac{(2i)!}{2^i i!}\cdot \frac{(n-2m+2i)!}{2^{\tfrac{n}{2} - m+i} (\tfrac{n}{2}- m+i)!} \cdot \frac{2^{\tfrac{n}{2}} (\tfrac{n}{2})!}{n!},
\end{align*}
where we use the convention that $0!=1$.
Using that for all $n$ we have 
\[\textstyle
n! = \sqrt{2\pi} n^{n+\tfrac{1}{2}} e^{-n} \exp\left(\sum_{k=n}^{\infty} a_k\right), \text{ where } a_k =\tfrac{1}{2} \int_0^1 \frac{x(1-x)}{(x+k)^2}\dx \leq \frac{1}{12k^2},
\]
we obtain
\begin{align*}
&\frac{(n-m)!}{(n-2m+2i)!} \cdot \frac{(2i)!}{2^i i!}\cdot \frac{(n-2m+2i)!}{2^{\tfrac{n}{2} - m+i} (\tfrac{n}{2}- m+i)!} \cdot \frac{2^{\tfrac{n}{2}} (\tfrac{n}{2})!}{n!} \\
&\leq \sqrt{n}\cdot 2^{\delta n} \cdot \exp\left( n \left( (1-\alpha)\log(1-\alpha) - \left(\tfrac{1}{2} - \alpha +\tfrac{i}{n} \right)\log  \left(\tfrac{1}{2} - \alpha +\tfrac{i}{n} \right)      \right)    \right).
\end{align*}
Now as $\delta\to 0$ (which implies $\alpha\to 1/2$) we have that 
\[
\left( (1-\alpha)\log(1-\alpha) - \left(\tfrac{1}{2} - \alpha +\tfrac{i}{n} \right)\log  \left(\tfrac{1}{2} - \alpha +\tfrac{i}{n} \right)      \right)    =- \frac{1}{2}\log2+ o(1).
\]
Using this together with the entropy bound
\[
\sum_{i\leq \delta n} {m\choose 2i} \leq 2^{mH(2\delta)},
\]
since $\delta<1/4$, 
with $H(p)$ being the entropy of a Bernoulli random variable with parameter $p$ and using the continuity of $H$ in $p$, gives
\begin{align*}
\pr{\exists \text{ less than $\delta n$ edges joining pairs of vertices of $A$}} \leq  2^{-n\left(\tfrac{1}{2} - C(\delta)\right)}, 
\end{align*}
where $C(\delta)$ is a constant only depending on $\delta$ satisfying $C(\delta)\to 0$ as $\delta\to 0$.
\end{proof}

\section*{Acknowledgements} The authors would like to thank Persi Diaconis, Bal\'azs Gerencs\'er, David Levin and Evita Nestoridi for useful discussions. Jonathan Hermon's research was supported by an NSERC grant.
Allan Sly's research was partially supported by NSF grant DMS-1855527, a Simons Investigator Grant and a MacArthur Fellowship. Perla Sousi's research was supported by the
Engineering and Physical Sciences Research Council: EP/R022615/1.

\bibliography{biblio}
\bibliographystyle{abbrv}

\end{document}